\documentclass[reqno,11pt]{amsart}

\usepackage[all]{xy}
\usepackage{amsthm, amssymb, amsmath, amscd}
\usepackage{braket}
\usepackage{geometry}
\usepackage[dvips]{graphicx, color}

\def\a{\alpha}

\def\e{\epsilon}

\def\l{\lambda}
\def\s{\sigma}
\def\t{\tau}

\def\NN{{\mathbb N}}
\def\PP{{\mathbb P}}
\def\QQ{{\mathbb Q}}

\def\ZZ{{\mathbb Z}}

\def\cal{\mathcal}

\def\cC{{\cal C}}

\def\cF{{\cal F}}
\def\cG{{\cal G}}

\def\cP{{\cal P}}

\def\cV{{\cal V}}

\def\fm{{\mathfrak m}}

\def\Aut{\operatorname{Aut}}

\def\add{\operatorname{add}}

\def\Coker{\operatorname{Coker}}

\def\CM{\operatorname {CM}}

\def\dim{\operatorname{dim}}

\def\Ext{\operatorname {Ext}}

\def\GKdim{\operatorname{GKdim}}

\def\gldim{\operatorname{gldim}}

\def\grmod{\operatorname{grmod}}
\def\GrMod{\operatorname{GrMod}}

\def\GrAut{\operatorname{GrAut}}

\def\H{\operatorname{H}}
\def\Hom{\operatorname {Hom}}

\def\id{\operatorname {id}}

\def\Im{\operatorname{Im}}

\def\injdim{\operatorname{injdim}}

\def\Ker{\operatorname {ker}}

\def\lin{\operatorname{lin}}

\def\mod{\operatorname{mod}}
\def\Mod{\operatorname{Mod}}

\def\Obj{\operatorname{Obj}}

\def\pdim{\operatorname{projdim}}

\def\Proj{\operatorname{Proj}}

\def\rank{\operatorname{rank}}

\def\sup{\operatorname{sup}}

\def\tors{\operatorname{tors}}
\def\Tors{\operatorname{Tors}}
\def\tails{\operatorname{tails}}
\def\Tails{\operatorname{Tails}}

\def\uCM{\underline {\operatorname{CM}}}

\def\R{\operatorname{\bf R}}
\def\RHom{\operatorname{{\bf R}Hom}}

\newcommand{\lotimes}{\otimes^{\bf{L}}}
\def\<{\langle}
\def\>{\rangle}

\def\NMF{\operatorname{NMF}}
\def\uNMF{\underline{\operatorname{NMF}}}
\def\TMF{\operatorname{TMF}}
\def\MF{\operatorname{MF}}

\def\TR{\operatorname{TR}}

\def\uTR{\underline{\operatorname{TR}}}

\def\sf{\nu}
\def\tw{\operatorname{tw}}
\def\hrank{\operatorname{hrank}}

\def\rnum#1{\expandafter{\romannumeral #1}}
\def\Rnum#1{\uppercase\expandafter{\romannumeral #1}}

\theoremstyle{plain} 
\newtheorem{theorem}{Theorem}[section]

\newtheorem{lemma}[theorem]{Lemma}
\newtheorem{proposition}[theorem]{Proposition}

\theoremstyle{definition}
\newtheorem{definition}[theorem]{Definition}
\newtheorem{example}[theorem]{Example}

\theoremstyle{remark}

\newtheorem{remark}[theorem]{Remark}

\numberwithin{equation}{section}

\begin{document}

\pagenumbering{arabic}

\title[Noncommutative Matrix Factorizations]
{Noncommutative Matrix Factorizations with an Application to Skew Exterior Algebras} 

\author{Izuru Mori}

\address{Department of Mathematics,
Faculty of Science,
Shizuoka University,
836 Ohya, Suruga-ku, Shizuoka 422-8529, Japan}
\email{mori.izuru@shizuoka.ac.jp}

\author{Kenta Ueyama}

\address{Department of Mathematics,
Faculty of Education,
Hirosaki University,
1 Bunkyocho, Hirosaki, Aomori 036-8560, Japan}
\email{k-ueyama@hirosaki-u.ac.jp}

\keywords {noncommutative matrix factorizations, totally reflexive modules, skew exterior algebras}

\subjclass[2010]{Primary: 16E05 \; Secondary: 16G50, 16S37, 16S38}

\thanks{The first author was supported by JSPS Grant-in-Aid for Scientific Research (C) 16K05097,
(C) 20K03510, and (B) 16H03923.
The second author was supported by JSPS Grant-in-Aid for Young Scientists (B) 15K17503 and JSPS Grant-in-Aid for Early-Career Scientists 18K13381.}

\begin{abstract}
Theory of matrix factorizations is useful to study hypersurfaces in commutative algebra.
To study noncommutative hypersurfaces, which are important objects of study in noncommutative algebraic geometry,
we introduce a notion of noncommutative matrix factorization for an arbitrary nonzero non-unit element of a ring. 
First we show that the category of noncommutative graded matrix factorizations is invariant under the operation called twist
(this result is a generalization of the result by Cassidy-Conner-Kirkman-Moore).
Then we give two category equivalences involving noncommutative matrix factorizations and totally reflexive modules
(this result is analogous to the famous result by Eisenbud for commutative hypersurfaces).
As an application,
we describe indecomposable noncommutative graded matrix factorizations over skew exterior algebras.
\end{abstract} 

\maketitle

\section{Introduction}
The theory of matrix factorizations was introduced by Eisenbud \cite{Ei} in 1980.
It is useful to study hypersurfaces in commutative algebra.
In fact, Eisenbud proved the following famous theorem.

\begin{theorem}[{\cite[Section 6]{Ei} (cf. \cite[Theorem 7.4]{Y})}] \label{thm.mot}
Let $S$ be a regular local ring, $f \in S$ a nonzero non-unit element, and $A=S/(f)$.
Let $\MF_S(f)$ denote the category of matrix factorizations of $f$ over $S$.
\begin{enumerate}
\item If $M$ is a maximal Cohen-Macaulay module over $A$ with no free summand, then the minimal free resolution of $M$ is obtained from a matrix factorization of $f$ over $S$, and hence it is periodic of period 1 or 2.
\item The factor category $\MF_S(f)/\add\{(1,f)\}$ is equivalent to the category $\CM(A)$ of maximal Cohen-Macaulay $A$-modules.
\item The factor category $\MF_S(f)/\add\{(1,f), (f,1)\}$ is equivalent to the stable category $\uCM(A)$ of maximal Cohen-Macaulay $A$-modules.
\end{enumerate}
\end{theorem}
Nowadays, matrix factorizations are related to several areas of mathematics, including representation theory of Cohen-Macaulay modules, singularity category, Calabi-Yau category, Khovanov-Rozansky homology, and homological mirror symmetry. 
In this paper, to investigate noncommutative hypersurfaces, which are important objects of study in noncommutative algebraic geometry (see \cite{SV}, \cite{KKZ}, \cite{MU}), we introduce a notion of noncommutative matrix factorization for an arbitrary nonzero non-unit element $f$ of a ring.

First, we will show that the category of noncommutative graded matrix factorizations $\NMF^{\ZZ}_S(f)$ is invariant under the operation called twist, which explains why the period of a noncommutative matrix factorization can be arbitrary (compare with Theorem \ref{thm.mot} (1)).

\begin{theorem}[{Theorem \ref{thm.nmft}}]
Let $S$ be a graded algebra, and $f\in S_d$ an element.  
If $\theta = \{\theta _i\}_{i \in \ZZ}$ is a twisting system on $S$ such that $\theta_i(f)=\l^if$ for some $0\neq \l\in k$ and for every $i\in \ZZ$, 
then $\operatorname {NMF}^{\ZZ}_S(f) \cong \operatorname {NMF}^{\ZZ}_{S^{\theta}}(f^{\theta})$. 
\end{theorem}

Suppose that $f$ is a regular normal element.
In this case, Cassidy-Conner-Kirkman-Moore \cite {CCKM} defined the notion of twisted matrix factorization of $f$,
and we will show that the category of noncommutative graded matrix factorizations of $f$ is equivalent to the category of twisted graded matrix factorizations of $f$ (Proposition \ref{prop.ctmf}),
so the above result is a generalization of \cite[Theorem 3.6]{CCKM}.

Next, we will show noncommutative graded and non-hypersurface analogues of Theorem \ref{thm.mot} (2) and (3).
Let $S$ be a graded algebra, $f\in S_d$ a regular normal element and $A=S/(f)$.
There are two types of trivial noncommutative matrix factorizations, $\phi_F$ and $_F\phi$, as defined in Definition \ref{def.trivnmf}.
We define $\cF =\{\phi_F \in \operatorname {NMF}^{\ZZ}_S(f) \mid F\; \textnormal{is free} \}$ and $\cG =\{\phi_F\oplus {_G\phi} \in \operatorname {NMF}^{\ZZ}_S(f) \mid F, G \; \textnormal{are free} \}$.
Let $\TR_S^\ZZ(A)$ denote the category of finitely generated graded totally reflexive $A$-modules which have finite projective dimension over $S$.
We define $\cP =\{P \in \grmod A \mid P\; \textnormal{is free}\}$.
With these notations, our result is stated as follows.

\begin{theorem}[{Theorem \ref{thm.m3}, Theorem \ref{thm.m4}}]
Let $S$ be a graded quotient algebra of a right noetherian connected graded regular algebra, $f\in S_d$ a regular normal element, and $A=S/(f)$.
Then the factor category $\NMF_S^{\ZZ}(f)/\cF$ is equivalent to $\TR^{\ZZ}_S(A)$.
Moreover, the factor category $\NMF_S^{\ZZ}(f)/\cG$ is equivalent to the factor category $\TR^{\ZZ}_S(A)/\cP$. 
\end{theorem}   
As an application of our theory of noncommutative matrix factorizations, we will describe indecomposable noncommutative matrix factorizations over skew exterior algebras (which are hardly noncommutative hypersurfaces). 
In particular, using techniques of noncommutative algebraic geometry,
we will show that, over a certain class of skew exterior algebras, there is a close relationship between noncommutative graded matrix factorizations and extensions of co-point modules (Theorems \ref{thm.copo} and \ref{thm.last}).  
An application to noncommutative hypersurfaces will be discussed in our subsequent paper \cite{MU2}.

\subsection{Basic Terminologies and Notations}

Throughout this paper, we fix a field $k$.
Unless otherwise stated, an \emph{algebra} means an algebra over $k$, and a \emph{graded ring} means an $\NN$-graded ring.

For a ring $S$, we denote by $\Mod S$ the category of right $S$-modules, and by $\mod S$ the full subcategory consisting of finitely generated modules.
We denote by $S^o$ the opposite ring of $S$.
We say that $S$ is \emph{regular} of dimension $n$ if $\gldim S =n < \infty$.

For a graded ring $S =\bigoplus_{i\in\NN} S_i$,
we denote by $\GrMod S$ the category of graded right $S$-modules, and by $\grmod S$ the full subcategory consisting of finitely generated modules.
Morphisms in $\GrMod S$ are right $S$-module homomorphisms preserving degrees.
For $M \in \GrMod S$ and $n \in \ZZ$, we define $M_{\geq n} := \bigoplus_{i\geq n} M_i \in \GrMod S$, and the \emph{shift} $M(n) \in \GrMod S$ by $M(n)_i = M_{n+i}$. For $M, N \in \GrMod S$, we write
$\Ext^i_S(M,N) :=\bigoplus_{n\in\ZZ}\Ext^i_{\GrMod S}(M,N(n))$ (by abuse of notation, we use the same symbol as in the ungraded case).

For a graded algebra $S =\bigoplus_{i\in\NN} S_i$, we say that $S$ is \emph{connected graded} if $S_0 = k$,
and we say that $S$ is \emph{locally finite} if $\dim_k S_i <\infty$ for all $i \in \NN$.
We denote by $\GrAut S$ the group of graded $k$-algebra automorphisms of $S$. 
We write $S^e$ for the enveloping algebra $S^o\otimes_k S$ of $S$.
If $S$ is a locally finite graded algebra and $M \in \grmod S$, then we define the \emph{Hilbert series} of $M$ by
$H_M (t) := \sum_{i \in \ZZ} (\dim_k M_i)t^i \in \ZZ[[t, t^{-1}]].$

A connected graded algebra $S$ is called an \emph{AS-regular} (resp. \emph{AS-Gorenstein}) algebra of dimension $n$ if
\begin{enumerate}
\item{} $\gldim S =n <\infty$ (resp. $\injdim_S (S) = \injdim_{S^o} (S)= n <\infty$), and
\item{} $\Ext^i_S(k ,S) \cong \Ext^i_{S^o}(k ,S) \cong
\begin{cases}
0 & \textnormal { if }\; i\neq n,\\
k(\ell) \; \textrm{for some}\; \ell \in \ZZ & \textnormal { if }\; i=n.
\end{cases}$
\end{enumerate}
Let $S$ be a noetherian AS-Gorenstein algebra of dimension $n$.
We define the \emph{local cohomology modules} of $M \in \grmod S$ by
$\H^i_\fm(M):= \lim _{n \to \infty} \Ext^i_S(S/S_{\geq n}, M)$.
It is well-known that $\H_{\fm}^i(S)=0$ for all $i\neq n$.
We say that $M \in \grmod S$ is \emph{graded maximal Cohen-Macaulay} if $\H_{\fm}^i(M)=0$ for all $i\neq n$.
We denote by $\CM^{\ZZ} (S)$ the full subcategory of $\grmod S$ consisting of graded maximal Cohen-Macaulay modules.

\section{Noncommutative Matrix Factorizations} 

\begin{definition}
Let $S$ be a ring and $f\in S$ an element.
A \emph{noncommutative right matrix factorization} of $f$ over $S$ is a sequence of right $S$-module homomorphisms $\{\phi^i:F^{i+1}\to F^i\}_{i\in \ZZ}$
where $F^i$ are free right $S$-modules of rank $r$ for some $r\in \NN$ such that there is a commutative diagram
\[\xymatrix@R=2pc@C=3pc{
F^{i+2} \ar[d]_{\cong} \ar[r]^{\phi^i\phi^{i+1}} &F^i \ar[d]^{\cong} \\
S^r \ar[r]^{f\cdot} &S^r 
}\]
for every $i\in \ZZ$.
A \emph{morphism }
$$\mu :\{\phi^i:F^{i+1}\to F^i\}_{i\in \ZZ}\to \{\psi^i:G^{i+1}\to G^i\}_{i\in \ZZ}$$
of noncommutative right matrix factorizations is a sequence of right $S$-module homomorphisms $\{\mu ^i:F^i\to G^i\}_{i\in \ZZ}$ such that the diagram 
\[\xymatrix@R=2pc@C=3pc{
F^{i+1} \ar[d]_{\mu ^{i+1}} \ar[r]^{\phi^i} &F^i \ar[d]^{\mu ^{i}} \\
G^{i+1} \ar[r]^{\psi^i} &G^{i}
}\]
commutes for every $i\in \ZZ$.
We denote by $\NMF_S(f)$ the category of noncommutative right matrix factorizations.  

Let $S$ be a graded ring and $f\in S_d$ a homogeneous element.
A \emph{noncommutative graded right matrix factorization} of $f$ over $S$ is a sequence of graded right $S$-module homomorphisms $\{\phi^i:F^{i+1}\to F^i\}_{i\in \ZZ}$
where $F^i$ are graded free right $S$-modules of rank $r$ for some $r\in \NN$ such that 
there is a commutative diagram 
\[\xymatrix@R=2pc@C=0.75pc{
F^{i+2} \ar[d]_{\cong} \ar[rrr]^{\phi^i\phi^{i+1}} &&&F^i \ar[d]^{\cong} \\
\bigoplus _{s=1}^rS(-m_{i+2,s})\ar@{=}[r] &\bigoplus _{s=1}^rS(-m_{is}-d) \ar[rr]^(0.54){f\cdot} &&\bigoplus _{s=1}^rS(-m_{is})  
}\]
for every $i\in \ZZ$. 
We can similarly define the category of noncommutative graded right matrix factorizations $\NMF^{\ZZ}_S(f)$.

We can also similarly define a \emph{noncommutative (graded) left matrix factorization} of $f$ over $S$.  
\end{definition} 

\begin{remark} \label{rem.la}
Let $S$ be a (graded) ring and $f\in S$ a (homogeneous) element. 
\begin{enumerate} 
\item{} The full subcategory $\NMF'_S(f)$ of $\NMF_S(f)$ consisting of $\{\phi^i:F^{i+1}\to F^i\}_{i\in \ZZ} \in \NMF_S(f)$ where $F^i=S^r$ for some $r\in \NN$ and $\phi^i\phi^{i+1}=f\cdot$ for every $i\in \ZZ$
may be a more straight noncommutative generalization of the category of commutative matrix factorizations (see Remark \ref{rem.cmf}).  By the definition of $\NMF_S(f)$, $\NMF'_S(f)$ is dense in $\NMF_S(f)$ so that they are equivalent categories, so, in this paper, we can and will use either category, which is more suitable depending on the context.  
\item{}  Let $\{\phi^i:F^{i+1}\to F^i\}_{i\in \ZZ}$ be a noncommutative right matrix factorization of $f$ over $S$ of rank $r$.
By (1), we often assume without loss of generality that $F^i=S^r$ and $\phi^i\phi^{i+1}=f\cdot$.
In this case, every $\phi^i$ is the left multiplication of a matrix $\Phi^i$ whose entries are elements in $S$, so that $\Phi^i\Phi^{i+1}=fE_r$ where $E_r$ is the identity matrix of size $r$.  
\item{}  Let $\{\phi^i:F^{i+1}\to F^i\}_{i\in \ZZ}$ be a noncommutative graded right matrix factorization of $f$ over $S$ of rank $r$ such that $F^i=\bigoplus _{s=1}^rS(-m_{is})$.
In this case, we may write $\phi^i=(\phi^i_{st})$ where $\phi^i_{st}:S(-m_{i+1, t})\to S(-m_{is})$ is the left multiplication of an element in $S_{m_{i+1, t}-m_{is}}$, so $\phi^i$ is the left multiplication of a matrix $\Phi^i$ whose entries are homogeneous elements in $S$, so that $\Phi^i\Phi^{i+1}=fE_r$ where $E_r$ is the identity matrix of size $r$.  
\item{} Two noncommutative (graded) right matrix factorizations $\phi$ and $\psi$ are isomorphic if and only if there are invertible matrices $P^i$ whose entries are in $S$ such that $\Psi^{i}=P^i\Phi^{i}(P^{i+1})^{-1}$ for every $i\in \ZZ$.
$$\begin{CD}
\cdots & @>\Phi ^{i+2}\cdot >> F^{i+2} @>\Phi ^{i+1}\cdot >>
F^{i+1} @>\Phi ^i\cdot >> F^i @>\Phi ^{i-1}\cdot >>& \cdots \\
& & & @VP^{i+2}\cdot VV @VP^{i+1}\cdot VV @VP^i\cdot VV \\
\cdots & @>\Psi^{i+2}\cdot>>G^{i+2} @>\Psi^{i+1}\cdot>> G^{i+1} @>\Psi^i\cdot >> G^i @>\Psi^{i-1}\cdot>>& \cdots
\end{CD}$$
\item{} Let $\{\phi^i:F^{i+1}\to F^i\}_{i\in \ZZ}$ be a sequence of right $S$-module homomorphisms between free right $S$-modules $F^i$ such that $\phi^i\phi^{i+1}=\l_if\cdot :F^{i+2}\to F^i$ for some unit element $\l_i\in S$ for every $i\in \ZZ$.  If $\{\psi^i:F^{i+1}\to F^i\}_{i\in \ZZ}$ is a sequence of right $S$-module homomorphisms inductively defined by 
\begin{align*}
\dots, &\psi^{-3}=\frac{\l_{-2}}{\l_{-1}\l_{-3}}\phi^{-3}, \psi^{-2}=\frac{\l_{-1}}{\l_{-2}}\phi^{-2}, \psi^{-1}=\frac{1}{\l_{-1}}\phi^{-1},  \\
&\psi^0=\phi ^0, \psi^1=\frac{1}{\l _0}\phi^1, \psi^2=\frac{\l_0}{\l_1}\phi^2, \psi ^3=\frac{\l_1}{\l_0\l_2}\phi^3, \psi ^4=\frac{\l_0\l_2}{\l_1\l_3}\phi^4, \dots, 
\end{align*}
then $\psi^i\psi^{i+1}=f\cdot :F^{i+2}\to F^i$.   
Since we have a commutative diagram
\[\xymatrix@R=2pc@C=6pc{
F^{i+2} \ar[d]_{\cong} \ar[r]^{\phi^i\phi^{i+1}=\l_if\cdot} &F^{i} \ar[d]^{\cong} \\
F^{i+2} \ar[r]^{\psi^i\psi^{i+1}=f\cdot} &F^{i},
}\]
it follows that $\phi=\{\phi^i\}_{i\in \ZZ} \in \NMF_S(f)$.
In particular, $\NMF_S(f)=\NMF_S(\l f)$ for every unit element $\l\in S$. 
\end{enumerate}
\end{remark}

\begin{example} \label{ex.2.3}
In practice, it is often easier to find $\phi\in \NMF_S(f)$ such that $\phi^i\phi^{i+1}=\l_if\cdot $ for some unit element $\l_i\in S$ than the one such that $\phi^i\phi^{i+1}=f\cdot $ for every $i\in \ZZ$.  For example, let $S=k\<x, y\>/(x^2, y^2)$, 
 $f=\a xy+yx\in S$ where $0\neq \a\in k$, $A=S/(f)$, and   
$M=A/(ax+by)A\in \mod A$ where $0\neq a, b\in k$.    The ``natural" choice of the differentials in the free resolution of $M$ is induced by $\phi^i=(ax+\a^iby)\cdot :S\to S$, but
$$\Phi^i\Phi^{i+1}=(ax+\a^iby)(ax+\a^{i+1}by)=a^2x^2+\a^iab(\a xy+yx)+\a^{2i+1}b^2y^2=\a^iabf.$$
If we define
\begin{align*}
\Psi^{2i} := \a^{-i}\Phi^{2i}=\a^{-i}ax+\a ^iby \quad \textrm{and}\quad
\Psi^{2i+1} := \a^{-i}a^{-1}b^{-1}\Phi^{2i+1}=\a^{-i}b^{-1}x+\a ^{i+1}a^{-1}y,
\end{align*} 
then 
\begin{align*}
\Psi ^{2i}\Psi ^{2i+1} & =(\a ^{-i}ax+\a^iby)(\a^{-i}b^{-1}x+\a^{i+1}a^{-1}y)\\
&=\a ^{-2i}ab^{-1}x^2+\a xy+yx+\a^{2i+1}ba^{-1}y^2=f, \\
\Psi ^{2i+1}\Psi ^{2i+2} & =(\a^{-i}b^{-1}x+\a^{i+1}a^{-1}y)(\a ^{-i-1}ax+\a^{i+1}by) \\
& =(\a ^{-2i-1}ab^{-1}x^2+\a xy+yx+\a^{2i+2}ba^{-1}y^2)=f, 
\end{align*}
so $\psi=\{\psi^i\}_{i\in \ZZ}\in \NMF_S(f)$, hence $\phi=\{\phi^i\}_{i\in \ZZ}\in \NMF_S(f)$.
\end{example} 

\begin{lemma} \label{lem.dmf} 
Let $S$ be a (graded) ring and $f\in S$ a (homogeneous) element.  If $\phi=\{\phi^i:F^{i+1}\to F^i\}_{i\in \ZZ}$ is a noncommutative (graded) right matrix factorization of $f$ over $S$, then 
\begin{align*}
& \Hom_S(\phi , S):=\{\Hom_S(\phi^{-i-1}, S):\Hom_S(F^{-i-1}, S)\to \Hom_S(F^{-i}, S)\}_{i\in \ZZ}
\end{align*} 
is a noncommutative (graded) left matrix factorization of $f$ over $S$. 
\end{lemma} 

\begin{proof} 
We have $\Hom_S(\phi ^{-i-1}, S)\Hom_S(\phi ^{-i-2}, S)=\Hom_S(\phi^{-i-2}\phi ^{-i-1}, S)=\Hom_S(f\cdot , S)=\cdot f:\Hom_S(F^{-i-2}, S)\to \Hom_S(F^{-i}, S)$, so the result follows. 
\end{proof} 

Let $S$ be a (graded) ring, $f\in S$ a (homogeneous) element, and $A=S/(f)$.
{\bf We tacitly assume that $f$ is a nonzero non-unit element for the rest of the paper, i.e., $\deg f\geq 1$ in the graded case}.  
For a noncommutative (graded) right matrix factorization $\phi$ of $f$ over $S$, we define the complex $C (\phi)$ of (graded) right $A$-modules by 
$$\begin{CD} \cdots @>\overline {\phi^2} >> \overline {F^2} @>\overline {\phi^1} >> \overline {F^1} @>\overline {\phi^0} >> \overline {F^0} @>\overline {\phi^{-1}}>> \overline {F^{-1}}  @>\overline {\phi^{-2}}>> \cdots \end{CD}$$
where we use $\overline{(-)}$ to denote $(-)\otimes_S A$.
Since $\overline {\phi^i}\; \overline {\phi ^{i+1}}=\overline {\phi^i\phi^{i+1}}=\overline {f\cdot}=0$, $C (\phi)$ is in fact a complex of (graded) right $A$-modules.
If $\mu :\{\phi^i:F^{i+1}\to F^i\}_{i\in \ZZ}\to \{\psi^i:G^{i+1}\to G^i\}_{i\in \ZZ}$ is a morphism of noncommutative (graded) right matrix factorizations, 
then the commutative diagram
$$\begin{CD}
F^{i+1} @>\phi ^i>> F^i \\
@V\mu ^{i+1}VV @VV\mu ^iV \\
G^{i+1} @>\psi^i>> G^i
\end{CD}$$
in $\Mod S$ ($\GrMod S$) induces the commutative diagram 
$$\begin{CD}
\overline {F^{i+1}} @>\overline {\phi ^i}>> \overline {F^i} \\
@V\overline{\mu ^{i+1}}VV @VV\overline {\mu ^i}V \\
\overline {G^{i+1}} @>\overline {\psi^i}>> \overline {G^i}
\end{CD}$$
in $\Mod A$ ($\GrMod A$) for every $i\in \ZZ$, so $C$ is a functor from the category of noncommutative (graded) right matrix factorizations to the category of cochain complexes of (graded) right $A$-modules.  

\section{Twisting Systems} 

Twist is a nice operation which is available only on graded rings.  In this section, we only deal with graded algebras. 

\begin{definition}[{\cite[Definition 2.1]{Zh}}]
 Let $S$ be a graded algebra.  A \emph{twisting system} on $S$ is a sequence $\theta=\{\theta_i\}_{i\in \ZZ}$ of graded $k$-linear automorphisms of $S$ such that $\theta_i(a\theta_j(b))=\theta_i(a)\theta_{i+j}(b)$ for every $i, j\in \ZZ$ and every $a\in S_j, b\in S$.  The \emph{twisted graded algebra} of $S$ by a twisting system $\theta $ is a graded algebra $S^{\theta}$ where $S^{\theta}=S$ as a graded $k$-vector space with the new multiplication $a^{\theta}b^{\theta}=(a\theta_{i}(b))^{\theta}$ for $a^{\theta}\in S^{\theta}_i, b^{\theta}\in S^{\theta}$.  Here we write $a^{\theta}\in S^{\theta}$ for $a\in S$ when viewed as an element of $S^{\theta}$ and the product $a\theta_{i}(b)$ is computed in $S$.  
\end{definition}  

Let $S$ be a graded algebra.  If $\s\in \GrAut S$, then $\{\s^i\}_{i\in \ZZ}$ is a twisting system of $S$.  In this case, we simply write $S^{\s}:=S^{\{\s^i\}}$.  
By \cite [Proposition 2.2, Proposition 2.4]{Zh}, {\bf we tacitly assume that $\theta_0=\id$ for every twisting system $\theta$ for the rest of the paper.}  We recall some results from \cite {Zh} which are needed in this paper.  

\begin{lemma}[{\cite [Proposition 2.5]{Zh}}] \label{lem.gra} 
Let $S$ be a graded algebra.  
If $\theta=\{\theta_i\}_{i\in \ZZ}$ is a twisting system on $S$, then  $\theta^{-1}=\{\theta_i^{-1}\}_{i\in \ZZ}$ is a twisting system on $S^{\theta}$ such that $(S^{\theta})^{\theta^{-1}}=S$. 
\end{lemma} 

Let $S$ be a graded algebra.  For $0\neq \l\in k$, the map $\e_{\l}:S\to S$ defined by $\e_{\l}(a)=\l^ia$ for $a\in S_i$ is a graded algebra automorphism of $S$.  We will leave the reader to check the following lemma. 

\begin{lemma} \label{lem.el}
Let $S$ be a graded algebra and $\theta=\{\theta_i\}_{i\in \ZZ}$ a twisting system on $S$.
For $0\neq \l\in k$, $\e_{\l}\theta :=\{ \e_{\l}^ i \theta_i \}_{i \in \ZZ}$ is a twisting system on $S$ and  the map $\phi:S^{\theta}\to S^{\e_{\l}\theta}$ defined by $\phi (a^{\theta})=\l^{i(i-1)/2}a^{\e_{\l}\theta}$ for $a\in S_i$ is an isomorphism of graded algebras. 
\end{lemma}  

\begin{remark} 
\label{rem.imp} 
Let $S$ be a graded algebra, $f\in S_d$, and $\theta$ a twisting system on $S$.
\begin{enumerate}
\item{} Suppose that $\theta_i(f)=\l_if$ for some $0\neq \l_i\in k$ for every $i\in \ZZ$.  Since $\theta_i(1)=1$ for every $i\in \ZZ$ by \cite [Proposition 2.2]{Zh},  
$$\l_j\l_if = \l_j\theta_i(f) = \theta _i(\l_jf) = \theta_i(\theta_j(f))= \theta_i(1)\theta_{i+j}(f)=\l_{i+j}f, $$
so $\l_i=\l^i$ for some $0\neq \l\in k$ for every $i\in \ZZ$.  
\item{} Suppose that $\theta_i(f)=\l^if$ for some $0\neq \l\in k$ for every $i\in \ZZ$. If $\theta'=\e_{\l^{-1/d}}\theta$, then 
$$\theta'_i(f)=\e_{\l^{-1/d}}^i\theta_i(f)=\l^{(-1/d)di}\l^if=f.$$
Moreover, there exists a graded algebra isomorphism $\phi:S^{\theta}\to S^{\theta'}$ such that $\phi(f^{\theta})=\l^{d(d-1)/2}f^{\theta'}$ by Lemma \ref{lem.el}, so
$$\NMF^{\ZZ}_{S^{\theta}}(f^{\theta})\cong \NMF^{\ZZ}_{S^{\theta'}}(\l ^{d(d-1)/2}f^{\theta'})=\NMF^{\ZZ}_{S^{\theta'}}(f^{\theta'})$$
by Remark \ref{rem.la} (5).
\item{}
By (1) and (2), we may replace the condition $\theta_i(f)=\l_if$ for some $0\neq \l_i\in k$ for every $i\in \ZZ$ by the simpler condition $\theta_i(f)=f$ for every $i\in \ZZ$ in some situations (see Theorem \ref{thm.cckm}).  
\end{enumerate}
\end{remark} 

Let $S$ be a graded algebra and $\theta $ a twisting system on $S$.  For $M\in \GrMod S$, we define $M^{\theta}\in \GrMod S^{\theta}$ by $M^{\theta}=M$ as a graded $k$-vector space with the new action $m^{\theta}a^{\theta}=(m\theta_i(a))^{\theta}$ for $m^{\theta}\in M^{\theta}_i, a^{\theta}\in S^{\theta}$.  For $\phi:M\to N$ in $\GrMod S$, we define $\phi ^{\theta}:M^{\theta}\to N^{\theta}$ in $\GrMod S^{\theta}$ by $\phi^{\theta}(m^{\theta})=\phi (m)^{\theta}$.  In fact, for $m^{\theta}\in M_i^{\theta}, a^{\theta}\in S^{\theta}$, 
\begin{align*}
\phi^{\theta} (m^{\theta}a^{\theta}) & = \phi ^{\theta}((m\theta_i(a))^{\theta})=\phi (m\theta_i(a))^{\theta} =(\phi (m)\theta_i(a))^{\theta}
=\phi (m)^{\theta}a^{\theta}=\phi^{\theta}(m^{\theta})a^{\theta},
\end{align*}
so $\phi^{\theta}$ is a graded right $S^{\theta}$-module homomorphism.  

\begin{theorem}[{\cite [Theorem 3.1]{Zh}}] \label{thm.Z}
For a graded algebra $S$ and a twisting system $\theta$ on $S$, $(-)^{\theta}:\GrMod S\to \GrMod S^{\theta}$ is an equivalence functor.  
\end{theorem} 

\begin{lemma} \label{lem.si} 
If $S$ is a graded algebra and $\theta$ is a twisting system on $S$, then the map $\theta_{m}:S^{\theta}(-m)\to S(-m)^{\theta}; \; a^{\theta}\mapsto \theta_{m}(a)^{\theta}$ is a graded right $S^{\theta}$-module isomorphism. 
\end{lemma} 

\begin{proof} This follows from the proof of \cite [Theorem 3.4]{Zh}. 
\end{proof} 

By Lemma \ref{lem.si}, a map $\phi:S(-m-m')\to S(-m)$ in $\GrMod S$ induces the map 
$$\begin{CD} \widetilde {\phi^{\theta}}:S^{\theta}(-m-m') @>\theta_{m+m'}>\cong > S(-m-m')^{\theta} @>\phi^{\theta}>> S(-m)^{\theta} @>\theta_{m}^{-1}>\cong > S^{\theta}(-m),\end{CD}$$
in $\GrMod S^{\theta}$.  The map $\widetilde {\phi^{\theta}}$ is defined by $\widetilde {\phi^{\theta}}(c^{\theta})=(\theta_{m}^{-1}\phi\theta_{m+m'}(c))^{\theta}$.  If $\psi:S(-m-m'-m'')\to S(-m-m')$ is a map in $\GrMod S$, then 
\begin{align*}
\widetilde {\phi^{\theta}}\widetilde {\psi^{\theta}}(c^{\theta}) 
& =\widetilde {\phi^{\theta}}((\theta_{m+m'}^{-1}\psi\theta_{m+m'+m''}(c))^{\theta})=(\theta_{m}^{-1}\phi\theta_{m+m'}\theta_{m+m'}^{-1}\psi\theta_{m+m'+m''}(c))^{\theta} \\
& = (\theta_{m}^{-1}\phi\psi\theta_{m+m'+m''}(c))^{\theta} = \widetilde {(\phi\psi)^{\theta}}(c^{\theta}),
\end{align*}
so $\widetilde {\phi^{\theta}}\widetilde {\psi^{\theta}}=\widetilde {(\phi\psi)^{\theta}}$.  

The map $\widetilde {\phi^{\theta}}$ must be the left multiplication of a certain element in $S^{\theta}_{m'}$.    In fact, if $\phi=a\cdot :S(-m-m')\to S(-m)$ for $a\in S_{m'}$, 
then   
\begin{align*}
\widetilde {\phi^{\theta}}(c^{\theta}) & 
= (\theta_{m}^{-1}\phi\theta_{m+m'}(c))^{\theta} 
= \theta_{m}^{-1}(a\theta_{m+m'}(c))^{\theta}
= \theta_{m}^{-1}(\theta_m(\theta_m^{-1}(a))\theta_{m+m'}(c))^{\theta}\\
&=\theta_{m}^{-1}\theta_m(\theta_m^{-1}(a)\theta_{m'}(c))^{\theta}
= (\theta_m^{-1}(a)\theta_{m'}(c))^{\theta}=\theta_m^{-1}(a)^{\theta}c^{\theta}, 
\end{align*}
so $\widetilde {\phi^{\theta}}$ is the left multiplication of $\theta_{m}^{-1}(a)^{\theta}\in S^{\theta}_{m'}$.  That is, we have the following commutative diagram
\[\xymatrix@R=2pc@C=6pc{
S(-m-m')^{\theta}  \ar[r]^(0.55){(a\cdot)^{\theta}} &S(-m)^{\theta}  \\
S^{\theta}(-m-m') \ar[r]^(0.55){\theta_m^{-1}(a)^{\theta}\cdot} \ar[u]^{\theta_{m+m'}}_{\cong}  &S^{\theta}(-m). \ar[u]_{\theta_{m}}^{\cong}
}
\]

If $\phi=(\phi_{ij})$ is a map of graded free right $S$-modules, 
then we define the map $\widetilde {\phi^{\theta}}=(\widetilde {\phi^{\theta} _{ij}})$ of graded free right $S^{\theta}$-modules. 

\begin{theorem} \label{thm.nmft} \label{thm.cckm}
Let $S$ be a graded algebra, and $f\in S_d$.  If $\theta$ is a 
twisting system on $S$ such that $\theta_i(f)=\l^if$ for some $0\neq \l\in k$ and for every $i\in \ZZ$, 
then 
$$\operatorname {NMF}^{\ZZ}_S(f)\to \operatorname {NMF}^{\ZZ}_{S^{\theta}}(f^{\theta}); \; \{\phi^i\}_{i \in \ZZ}\mapsto \{(\phi ^i)^{\theta}\}_{i \in \ZZ}$$ 
is an equivalence functor. 
\end{theorem} 

\begin{proof}
By Remark \ref{rem.imp} (3), we may assume that $\theta_i(f)=f$ for every $i\in \ZZ$.
Since $(-)^{\theta}:\GrMod S\to \GrMod S^{\theta}$ is a functor by Theorem \ref{thm.Z}, and $\theta_i^{-1}(f)^\theta = f^\theta$ for every $i\in \ZZ$,
if $\{\phi^i\}_{i \in \ZZ}\in \NMF^{\ZZ}_S(f)$, then the commutative diagram 
$$
\xymatrix@R=2pc@C=2pc{
F^{i+2} \ar[d]_{\cong} \ar[r]^{\phi^i\phi^{i+1}} &F^i \ar[d]^{\cong} \\
\bigoplus _{s=1}^rS(-m_{is}-d) \ar[r]^(0.54){f\cdot} &\bigoplus _{s=1}^rS(-m_{is})  
}
$$
induces the commutative diagram
$$
\xymatrix@R=2pc@C=8pc{
(F^{i+2})^{\theta} \ar[d]_{\cong}
\ar[r]^{(\phi^i)^\theta (\phi^{i+1})^\theta\ =\ (\phi^i \phi^{i+1})^\theta}
&(F^i)^{\theta} \ar[d]^{\cong} \\
\bigoplus _{s=1}^rS(-m_{is}-d)^\theta
\ar[r]^(0.54){(f\cdot)^\theta}
&\bigoplus _{s=1}^rS(-m_{is})^\theta\\
\bigoplus _{s=1}^rS^\theta(-m_{is}-d) \ar[u]^{\bigoplus_{s=1}^r \theta_{m_{is}+d}}_{\cong} \ar[r]^(0.54){
 f^\theta \cdot} &\bigoplus _{s=1}^rS^\theta(-m_{is}), \ar[u]_{\bigoplus_{s=1}^r \theta_{m_{is}}}^{\cong}
}
$$
so $\{(\phi ^i)^{\theta}\}_{i \in \ZZ}\in \NMF^{\ZZ}_{S^{\theta}}(f^{\theta})$.

Since $(-)^{\theta}:\GrMod S\to \GrMod S^{\theta}$ is a functor, a commutative diagram 
$$\begin{CD}
F^{i+1} @>\phi ^i>> F^i \\
@V\mu ^{i+1}VV @VV\mu ^iV \\
G^{i+1} @>\psi^i>> G^i
\end{CD}$$
in $\GrMod S$ induces a commutative diagram 
$$\begin{CD}
(F^{i+1})^{\theta} @>(\phi ^i)^{\theta}>> (F^i)^{\theta} \\
@V(\mu ^{i+1})^{\theta}VV @VV(\mu ^i)^{\theta}V \\
(G^{i+1})^{\theta} @>(\psi^i)^{\theta}>> (G^i)^{\theta} 
\end{CD}$$
in $\GrMod S^{\theta}$, so 
$$\operatorname {NMF}^{\ZZ}_S(f)\to \operatorname {NMF}^{\ZZ}_{S^{\theta}}(f^{\theta}); \; \{\phi^i\}_{i \in \ZZ}\mapsto \{(\phi ^i)^{\theta}\}_{i \in \ZZ}$$
is a functor. By Lemma \ref{lem.gra}, 
$$\operatorname {NMF}^{\ZZ}_{S^{\theta}}(f^{\theta})\to \operatorname {NMF}^{\ZZ}_{(S^{\theta})^{\theta^{-1}}}((f^{\theta})^{\theta^{-1}})\cong \operatorname {NMF}^{\ZZ}_{S}(f); \; \{\psi^i\}_{i \in \ZZ}\mapsto \{(\psi ^i)^{\theta^{-1}}\}_{i \in \ZZ}$$ is the inverse functor.  
\end{proof} 

\begin{remark}  If $f$ is a ``regular normal" element and $\nu$ is the ``normalizing" automorphism of $f$, then Cassidy-Conner-Kirkman-Moore \cite {CCKM} defined the notion of twisted matrix factorization of $f$ (see the next section).  In this case, we will show that the category of noncommutative graded matrix factorizations of $f$ is equivalent to the category of twisted graded matrix factorizations of $f$ (Proposition \ref{prop.ctmf}), so Theorem \ref{thm.cckm} was proved in \cite [Theorem 3.6]{CCKM} with the technical condition $\theta_i\nu^{-1} \theta_d=\nu^{-1}\theta_{i+d}$ for every $i\in \ZZ$.  Since Theorem \ref{thm.cckm} does not require $f$ to be regular normal, it is a generalization of \cite [Theorem 3.6]{CCKM}.
\end{remark} 

Let $S$ be a graded algebra, $f\in S_d$, and $\theta$ a 
twisting system on $S$ such that $\theta_i(f)=f$ for all $i\in \ZZ$.  If $\phi=\{\phi^i:F^{i+1}\to F^i\}_{i \in \ZZ}\in \NMF^{\ZZ}_S(f)$ such that $F^i\cong S(-m_i)^r$, then 
$\{\widetilde {(\phi^i)^{\theta}}\}_{i \in \ZZ}=\{\theta_{m_i}^{-1}(\phi^i)^{\theta}\}_{i \in \ZZ}\in \NMF^{\ZZ}_{S^{\theta}}(f^{\theta})$.  
In particular, if $F^0\cong S^r$, $d=2$ and $\theta =\{\s^i\}_{i \in \ZZ}$ for $\s\in \GrAut S$, then $m_i=i$ for all $i\in \ZZ$, so 
$\{\widetilde {(\phi^i)^{\s}}\}_{i \in \ZZ}=\{\s^{-i}(\phi^i)^{\s}\}_{i \in \ZZ}\in \NMF^{\ZZ}_{S^{\s}}(f^{\s})$.  The following simple example illustrates the difference between $\{(\phi^i)^{\theta}\}_{i \in \ZZ}$ and $\{\widetilde {(\phi^i)^{\theta}}\}_{i \in \ZZ}$.  

\begin{example} 
Let $S=k[x, y]$ with $\deg x=\deg y=1$, $f=xy\in S_2$, and $\s\in \GrAut S$ such that $\s(x)=y, \s(y)=x$ so that $\s(f)=f$.  The noncommutative graded right matrix factorization of $f$ over $S$ 
$$\begin{CD}
\phi:\cdots @>y\cdot >> S(-3) @>x\cdot >> S(-2) @>y\cdot >> S(-1) @>x\cdot >> S @>y\cdot >> \cdots 
\end{CD}$$
induces the commutative diagram 
$$\begin{CD}
\phi^{\s}: \cdots @>(y\cdot)^{\s} >> S(-3)^{\s} @>(x\cdot )^{\s}>> S(-2)^{\s} @>(y\cdot)^{\s} >> S(-1)^{\s} @>(x\cdot)^{\s} >> S^{\s} @>(y\cdot)^{\s} >> \cdots \\
& & @V\cong VV @V\cong VV @V\cong VV @V\cong VV  \\
\widetilde {\phi^{\s}}:\cdots @>x^{\s}\cdot >> S^{\s}(-3) @>x^{\s}\cdot >> S^{\s}(-2) @>x^{\s}\cdot >> S^{\s}(-1) @>x^{\s}\cdot >> S^{\s} @>x^{\s}\cdot >> \cdots, \\
\end{CD}$$
It follows that $\widetilde {\phi^{\s}}$ is a noncommutative graded right matrix factorization of $f^{\s}=(xy)^{\s}=x^{\s}x^{\s}$ over $S^{\s}$ in the strict sense, but 
the above commutative diagram shows that ${\phi^{\s}}$ is also (isomorphic to) a noncommutative graded right matrix factorization of $f^{\s}$ over $S^{\s}$. 
\end{example} 
 
With the above example understood, we often identify $\{(\phi^i)^{\theta}\}_{i\in \ZZ}$ with $\{\widetilde {(\phi^i)^{\theta}}\}_{i\in \ZZ}$.  

\begin{example}
Let $S=k\<x, y\>/(x^2, y^2)$ with $\deg x=\deg y=1$, $f=xy+yx\in S_2$, and $A=S/(f)$.
If $\s'=\begin{pmatrix} \a & 0 \\ 0 & 1 \end{pmatrix}\in \GrAut S$, then $\s'(f)=\a f$.  
If we modify $\s'$ as $\s=\begin{pmatrix} \a^{1/2} & 0 \\ 0 & \a^{-1/2} \end{pmatrix}=\a^{-1/2}\s'\in \GrAut S$, then $S^{\s'}\cong S^{\s}$ and $\s(f)=f$.  

For $0\neq a, b\in k$, $\{\phi^i\}_{i\in \ZZ}$ defined by $\Phi^{2i}=ax+by, \Phi^{2i+1}=b^{-1}x+a^{-1}y$ is a noncommutative graded right matrix factorization of $f$ over $S$, so 
$\{\s^{-i}(\phi^i)\}_{i\in \ZZ}$ defined by 
$\Psi^{2i}=\s^{-2i}(\Phi^{2i})=\a^{-i}ax+\a ^{i}by, \Psi^{2i+1}=\s^{-2i-1}(\Phi^{2i+1})=\a^{-(2i+1)/2}b^{-1}x+\a ^{(2i+1)/2}a^{-1}y$, that is, 
$$\Psi ^i=\begin{cases} \a^{-i/2}ax+\a^{i/2}by & \textnormal { if $i$ is even, } \\ 
\a^{-i/2}b^{-1}x+\a^{i/2}a^{-1}y & \textnormal { if $i$ is odd, } 
\end{cases}$$
is a noncommutative graded right matrix factorization of $f^{\s}=\a ^{1/2}xy+\a ^{-1/2}yx$ over $S^{\s}$.  (Compare with Example \ref{ex.2.3}.)  
\end{example} 

\section{Regular Normal Elements} 

\begin{definition}
Let $S$ be a ring, and $f\in S$.  
\begin{enumerate}
\item{} We say that $f$ is \emph{regular} if, for every $a\in S$, $af=0$ or $fa=0$ implies $a=0$. 
\item{} We say that $f$ is \emph{normal} if $Sf=fS$. 
\end{enumerate}
\end{definition} 

\begin{remark}
Let $S$ be a (graded) ring and $f\in S$ a (homogeneous) element (of degree $d$). 
\begin{enumerate}
\item{} $f$ is regular if and only if the map $f\cdot :S\to S$ ($f\cdot :S(-d)\to S$) is an injective (graded) right $S$-module homomorphism and $\cdot f:S\to S$ ($\cdot f:S(-d)\to S$) is an injective (graded) left $S$-module homomorphism. 
\item{} $f$ is regular normal if and only if there exists a unique (graded) ring automorphism $\nu_f$ of $S$ such that $af=f\nu_f(a)$ for $a\in S$.  
We call $\nu := \nu_f$ the \emph{normalizing automorphism} of $f$.  
\end{enumerate}
\end{remark} 

\begin{remark}
Let $S$ be a noetherian AS-regular algebra and $f \in S_d$ a regular normal element.
Then $A= S/(f)$ is a noetherian AS-Gorenstein algebra, which is called a \emph{noncommutative hypersurface} of degree $d$.
\end{remark}

Let $S$ be a ring, $f\in S$, and $\{\phi^i:F^{i+1}\to F^i\}_{i\in \ZZ} \in \NMF_S(f)$.
From now on, we usually assume without loss of generality that $F^i=S^r$ for some $r\in \NN$ and $\phi^i\phi^{i+1}=f\cdot$ for every $i\in \ZZ$, that is, we implicitly use the category $\NMF'_S(f)$ defined in Remark \ref{rem.la} (1) in place of $\NMF_S(f)$.

If $\Phi=(a_{st})$ is a matrix whose entries are (homogeneous) elements in $S$, and $\s$ is a (graded) algebra automorphism of $S$, then we write $\s(\Phi)=(\s(a_{st}))$.  If $\phi$ is a (graded) right $S$-module homomorphism 
given by the left multiplication of $\Phi$, then we write $\s(\phi)$ for the (graded) right $S$-module homomorphism given by the left multiplication of $\s(\Phi)$.  
Let $f\in S$ be a (homogeneous) regular normal element.  Since $f\sf(\Phi)=(f\sf(a_{st}))=(a_{st}f)=\Phi f$, we have $f\sf(\phi)=\phi f$
($f(\sf(\phi)(-d))=\phi f$ where $d=\deg f$).  

\begin{theorem} \label{thm.nmf} 
Let $S$ be a (graded) ring and $f\in S$ a (homogeneous) regular normal element (of degree $d$). 
\begin{enumerate}
\item{} If $\phi$ is a noncommutative (graded) right matrix factorization of $f$ over $S$, then $\phi^{i+2}=\sf(\phi^i)$ ($\phi^{i+2}=\sf(\phi^i)(-d)$) for every $i\in \ZZ$.  It follows that $\phi$ is uniquely determined by $\phi^0$ and $\phi^1$. 
\item{} Let $\phi^0, \phi^1$ be (graded) right $S$-module homomorphisms between (graded) free right $S$-modules of rank $r$ such that $\phi^0\phi^1=f\cdot$.  If  $\phi^0$ is injective, then there exists a unique noncommutative (graded) right matrix factorization $\phi$ extending $\phi^0, \phi^1$. 
\item{} If $\mu:\phi\to \psi$ is a morphism of noncommutative (graded) right matrix factorizations of $f$ over $S$, then $\mu^{i+2}=\sf(\mu^i)$ ($\mu^{i+2}=\sf(\mu^i)(-d)$) for every $i\in \ZZ$.  It follows that $\mu$ is uniquely determined by $\mu^0$ and $\mu^1$. 
\item{} Let $\phi, \psi$ be noncommutative (graded) right matrix factorizations of $f$ over $S$ and $\mu^0:F^0\to G^0, \mu^1:F^1\to G^1$ are (graded) right $S$-module homomorphisms such that $\mu^0 \phi^0=\psi^0\mu^1 $, then there exists a unique morphism $\mu:\phi\to \psi$ extending $\mu ^0, \mu ^1$.  
\end{enumerate}
\end{theorem} 

\begin{proof} 
(1) Since $\phi^i\phi^{i+1}=\phi^{i+1}\phi^{i+2}=f\cdot $ for every $i\in \ZZ$, 
$$f\sf(\phi^i)=\phi^i f=\phi^i\phi^{i+1}\phi^{i+2}=f\phi^{i+2}.$$  
Since $f$ is regular (the map $f\cdot$ is injective), $\phi^{i+2}=\sf(\phi^i)$. 

(2) Let $\phi^0:G\to F, \phi ^1:F\to G$ such that $\phi^0\phi^1=f\cdot :F\to F$.
Since $\phi^0\phi^1\sf(\phi^0)=f\sf(\phi^0)=\phi^0 f$, and $\phi^0$ is injective, $\phi^1\sf(\phi^0)=f\cdot :G\to G$, so define $\phi^2=\sf(\phi^0)$.  Since 
$\sf^{-1}(\phi^1)\phi^0=\sf^{-1}(\phi^1\sf(\phi^0))=\sf^{-1}(f)\cdot =f\cdot : G\to G$,  define $\phi^{-1}=\sf^{-1}(\phi^1)$.
Inductively we can construct a noncommutative (graded) right matrix factorization $\phi$ extending $\phi^0, \phi^1$.
The uniqueness follows from (1).  

(3) By (1), $\sf^{-1}(\psi^{i+1})\psi^i= \psi^{i-1}\psi^i =f\cdot$, so 
\begin{align*}
f\psi^{i+1}\mu^{i+2} & = f\mu ^{i+1}\phi^{i+1} = \sf^{-1}(\psi^{i+1})\psi^i\mu ^{i+1}\phi^{i+1}\\
&= \sf^{-1}(\psi^{i+1})\mu ^i\phi^i\phi^{i+1} = \sf^{-1}(\psi^{i+1})\mu ^if = f\psi^{i+1}\sf(\mu ^i).
\end{align*}
Since $f\psi^{i+1}$ is injective, $\mu^{i+2}=\sf(\mu^i)$.  

(4) Since $\psi ^0\psi^1\sf(\mu ^0)=f\sf(\mu ^0)=\mu ^0f=\mu ^0\phi^0\phi^1=\psi^0\mu ^1\phi^1$ and $\psi^0$ is injective, $\psi^1\sf(\mu^0)=\mu^1\phi^1$, so define $\mu^2=\sf(\mu^0)$.
Since $\sf^{-1}(\mu^1)\phi^{-1}=\sf^{-1}(\mu^1)\sf^{-1}(\phi^1)=\sf^{-1}(\mu^1\phi^1)=\sf^{-1}(\psi^1\sf(\mu^0))=\sf^{-1}(\psi^1)\mu^0=\psi^{-1}\mu^0$, define $\mu^{-1}=\sf^{-1}(\mu^1)$.
Inductively we can construct a morphism $\mu:\phi\to \psi$ extending $\mu^0, \mu^1$.
The uniqueness follows from (3).  
\end{proof} 

\begin{remark} \label{rem.cmf} 
If $S$ is a commutative (graded) ring and $f$ is a (homogeneous) regular element (of degree $d$), then $f$ is central (i.e., $\nu=\id$), so a noncommutative (graded) matrix factorization $\phi$ of $f$ over $S$ determines and is determined by an ordered pair of (graded) $S$-module homomorphisms $(\phi^0: F \to G, \phi^1:G\to F)$ where $F, G$ are (graded) free $S$-modules of rank $r$ for some $r \in \NN$ such that $\phi^1\phi^0 =f \cdot $ and $\phi^0\phi^1 = f \cdot $ by Theorem \ref{thm.nmf} (1), (2).
Therefore, we see that $\NMF_S(f)$ ($\NMF^{\ZZ}_S(f)$) is equivalent to the category $\MF_S(f)$ ($\MF^{\ZZ}_S(f)$) of usual (graded) matrix factorizations of $f$ over $S$. (Note that morphisms of these categories also agree by Theorem \ref{thm.nmf} (3), (4).)
\end{remark}

In the case that $f\in S$ is a homogeneous regular normal element, the notion of twisted graded matrix factorization was defined in \cite {CCKM}.  We will show that a twisted graded matrix factorization and a noncommutative graded matrix factorization are the same if $f\in S$ is a homogeneous regular normal element.  

Let $S$ be a connected graded algebra.
For $\s \in \GrAut S$ and $M \in \GrMod S$, we define $M_\s \in \GrMod S$ where $M_\s = M$ as a graded $k$-vector space with the new action $m*a = m\s(a)$ for $m \in M, a \in S$.
If $\phi: M \to N$ is a graded right $S$-module homomorphism, then $\phi_{\s} :M_\s \to N_\s; \; m \mapsto \phi(m)$ is also a graded right $S$-module homomorphism.
Let $f \in S_d$ be a regular normal element, and $\nu$ the normalizing automorphism of $f$.
Then we write $M_{\tw} = M_{\nu^{-1}}(-d)$. Note that a graded right $S$-module homomorphism $\phi: M \to N$ naturally induces a graded right $S$-module homomorphism $\phi_{\tw}: M_{\tw} \to N_{\tw}$.
It is easy to see that the map $\nu:S(-m)_{\tw}=S_{\nu^{-1}}(-m-d) \to S(-m-d); \; a \mapsto \nu(a)$ is a graded right $S$-module homomorphism for every $m\in \ZZ$, so,  
for a graded free right $S$-module $F=\bigoplus _{s=1}^rS(-m_{s})$, $\nu$ induces a graded right $S$-module homomorphism $F_{\tw} \to F(-d)$, which is denoted by $\nu$ by abuse of notation.


\begin{definition}[\textnormal{\cite[Definitions 2.2, 2.5]{CCKM}}]
Let $S$ be a locally finite connected graded algebra, $f \in S_d$ a regular normal element, and $\nu$ the normalizing automorphism of $f$.
A \emph{twisted graded right matrix factorization} of $f$ over $S$ is an ordered pair of graded right $S$-module homomorphisms
\[ (\psi :F \to G, \; \t: G_{\tw} \to F) \]
where $F, G$ are graded free right $S$-modules of rank $r$ for some $r \in \NN$ such that 
$\psi \t= \cdot f$ and $\t \psi_{\tw}= \cdot f$. 
A \emph{morphism} $(\psi, \t) \to (\psi', \t')$ of twisted graded right matrix factorizations is a pair $(\mu_G, \mu_F)$ of graded right $S$-module homomorphisms such that the diagram 
\[\xymatrix@R=2pc@C=3pc{
F \ar[d]_{\mu_F} \ar[r]^{\psi} &G \ar[d]^{\mu_G} \\
F' \ar[r]^{\psi'} &G'
}\]
commutes. We denote by $\TMF_S^{\ZZ}(f)$ the category of twisted graded right matrix factorizations.
\end{definition}

\begin{remark}  \label{rem.tmf}
The comment in \cite [Section 1]{CCKM} is misleading.  The tensor algebra $T:=T(V)$ of an $n$-dimensional vector space $V$ over $k$ is a locally finite connected graded algebra.  Since the minimal free resolution of $k$ over $T$ is given by $0\to T(-1)^n\to T\to k$, so, in particular,  $T(-1)^n\to T$ is injective, however, if $n>1$, then $\rank  T(-1)^n=n>\rank T=1$.
For this reason, it is unclear that $(\psi, \t)\in \TMF_S^{\ZZ}(f)$ implies $\rank F=\rank G$ in general as claimed in \cite [Section 2]{CCKM}, so we impose the condition $\rank F=\rank G$ in the definition of a twisted graded matrix factorization in this paper.  (This problem will be avoided by assuming that $S$ is a graded quotient algebra of a right noetherian connected graded regular algebra in Section 6.)   
\end{remark} 

\begin{proposition} \label{prop.ctmf}
If $S$ is a locally finite connected graded algebra, and $f \in S_d$ is a regular normal element, 
then $\NMF_S^{\ZZ}(f) \cong \TMF_S^{\ZZ}(f)$.
\end{proposition} 

\begin{proof}
Let $(\psi :F \to G, \; \t: G_{\tw} \to F)\in \TMF_S^{\ZZ}(f)$.
By similar arguments to Remark \ref{rem.la} (1), we may assume that
$F = \bigoplus _{s=1}^rS(-m_{s})$ and $G=\bigoplus _{s=1}^rS(-l_{s}) $ without loss of generality.
Then we have the following commutative diagram
\[\xymatrix@R=2.25pc@C=5pc{
G_{\tw} \ar[d]_{\nu}^{\cong} \ar[r]^{\t}  &F \ar[r]^{\psi} \ar@{=}[d] & G \ar@{=}[d]\\
G(-d)\ar[r]^{\t \nu^{-1}} &F \ar[r]^{\psi} &G.
}\]
Put $\phi^0:= \psi$ and $\phi^1:= \t \nu^{-1}$. For $x \in G(-d)$, it follows that $\phi^0\phi^1(x) = \psi\t \nu^{-1}(x) = \nu^{-1}(x)f = fx$,
so $\phi^0\phi^1 = f\cdot$. By Theorem \ref{thm.nmf} (2), we obtain a unique $\phi_{(\psi,\t)} := \{ \phi^i\}_{i \in \ZZ}\in \NMF_S^{\ZZ}(f)$ extending $\phi^0, \phi^1$.
By Theorem \ref{thm.nmf} (4), we see that a morphism $(\psi, \t) \to (\psi', \t')$ in $\TMF_S^{\ZZ}(f)$ induces a unique morphism  $\phi_{(\psi,\t)} \to \phi_{(\psi',\t')}$ in $\NMF_S^{\ZZ}(f)$. Hence this construction defines a functor ${\mathfrak F}:\TMF_S^{\ZZ}(f) \to \NMF_S^{\ZZ}(f)$.

Conversely, if $\{\phi^i:F^{i+1}\to F^i\}_{i\in \ZZ}\in \NMF_S^{\ZZ}(f)$ such that 
$F^1 = \bigoplus _{s=1}^rS(-m_{s})$ and $F^0=\bigoplus _{s=1}^rS(-l_{s})$, then we have the following commutative diagram
\[\xymatrix@R=2.25pc@C=5pc{
F(-d) \ar[d]_{\nu^{-1}}^{\cong} \ar[r]^{\phi^2} &G(-d) \ar[d]_{\nu^{-1}}^{\cong} \ar[r]^{\phi^1}  &F \ar[r]^{\phi^0} \ar@{=}[d] & G \ar@{=}[d]\\
F_{\tw} \ar[r]^{\phi^0_{\tw}} &G_{\tw}\ar[r]^{\phi^1 \nu} &F \ar[r]^{\phi^0} &G
}\]
where $G := F^0, F:=F^1$. In fact, for $x \in F(-d)$, we have $\nu^{-1}\phi^2(x) = \nu^{-1}(\nu(\phi^0)(x)) = \nu^{-1}(\nu(\Phi^0)x) = \Phi^0 \nu^{-1}(x) = 
\phi^0_{\tw}\nu^{-1}(x)$ by Theorem \ref{thm.nmf} (1).
Put $\psi:= \phi^0$ and $\t:= \phi^1 \nu$.
For $y \in G_{\tw}, x \in F_{\tw}$, it follows that $\psi\t(y) = \phi^0\phi^1\nu (y) = f\nu(y) =yf$ and $\t\psi_{\tw}(x)= \phi^1 \nu \phi^0_{\tw}(x)= \phi^1\phi^2 \nu(x)= f\nu(x) =xf$,
so $\psi \t= \cdot f$ and $\t\psi_{\tw}= \cdot f$. Thus we obtain $(\psi,\t)\in \TMF_S^{\ZZ}(f)$.
This construction defines a functor ${\mathfrak G}:\NMF_S^{\ZZ}(f) \to \TMF_S^{\ZZ}(f)$, which is an inverse to $\mathfrak F$.
\end{proof}

\begin{remark}
The category $\NMF_S^{\ZZ}(f)$ is hardly abelian.
Presumably, the correct statement of \cite [Proposition 3.1]{CCKM} is ``$\TMF_S^{\ZZ}(f)$ is an additive category" instead of an abelian category. 
\end{remark} 

Let $S$ be a locally finite connected graded algebra, $f \in S_d$ a regular normal element, and $A=S/(f)$.
By \cite[Proposition 2.4]{CCKM}, for $(\psi :F \to G, \; \t: G_{\tw} \to F)\in \TMF_S^{\ZZ}(f)$, we have the complex $C'(\psi, \t)$ of graded right $A$-modules defined by
\[\xymatrix@R=2pc@C=2pc{
\cdots \ar[r]^(0.3){\overline{\t_{\tw^{i+1}}}} &\overline{F_{\tw^{i+1}}} \ar[r]^(0.5){\overline{\psi_{\tw^{i+1}}}} &\overline{G_{\tw^{i+1}}}  \ar[r]^(0.6){\overline{\t_{\tw^i}}} &\overline{F_{\tw^i}} \ar[r]^{\overline{\psi_{\tw^i}}} &\overline{G_{\tw^i}} \ar[r] &\cdots.
}\]

\begin{proposition} \label{prop.tmf}
Let $S$ be a locally finite connected graded algebra, and $f \in S_d$ a regular normal element.
For every $\phi=\{\phi^i:F^{i+1}\to F^i\}_{i\in \ZZ}\in \NMF_S^{\ZZ}(f)$,
there exists $(\phi^0, \phi^1\nu)\in \TMF_S^{\ZZ}(f)$ such that $C(\phi) \cong C'(\phi^0, \phi^1\nu)$.
\end{proposition}

\begin{proof} 
By the proof of Proposition \ref{prop.ctmf}, $(\phi^0, \phi^1\nu)\in \TMF_S^{\ZZ}(f)$. Moreover, we can check that
\[\xymatrix@R=2.25pc@C=5pc{
F^{2i+2} \ar[d]_{\nu^{-i-1}}^{\cong} \ar[r]^{\phi^{2i+1}} &F^{2i+1} \ar[d]_{\nu^{-i}}^{\cong} \ar[r]^{\phi^{2i}}  &F^{2i} \ar[d]_{\nu^{-i}}^{\cong}\\
G_{\tw^{i+1}} \ar[r]^{(\phi^1 \nu)_{\tw^{i}}} &F_{\tw^{i}}\ar[r]^{(\phi^0)_{\tw^{i}}} &G_{\tw^{i}}
}\]
commutes for every $i \in \ZZ$, 
so $C(\phi) \cong C'(\phi^0, \phi^1\nu)$.
\end{proof}

\begin{definition} Let $S$ be a (graded) ring and $f\in S$ a (homogeneous) element.
The \emph{period} of a noncommutative (graded) right matrix factorization $\phi=\{\phi^i\}_{i\in \ZZ}$ of $f$ over $S$, denoted by $|\phi|$, is defined to be 
the smallest positive integer $\ell\in \NN^+$ such that $\Coker (\overline {\phi^{\ell}})\cong \Coker (\overline {\phi^0})$
($\Coker (\overline {\phi^{\ell}})\cong \Coker (\overline {\phi^0})(-m)$ for some $m\in \ZZ$).
\end{definition}   

The lemma below follows from Theorem \ref{thm.nmf} (1), which also follows from the combination of Proposition \ref{prop.tmf} and \cite [Proposition 2.12]{CCKM}. (See Theorem \ref{thm.mot} (1).)

\begin{lemma} 
Let $S$ be a (graded) ring and $f\in S$ a (homogeneous) regular normal element. 
If $\phi$ is a noncommutative (graded) right matrix factorization of $f$ over $S$, then $|\phi|\leq 2|\sf|$.  In particular, if $f$ is central, then $|\phi|$ is 1 or 2.  
\end{lemma} 

\begin{example} Let $S=k\<x, y\>/(x^2, y^2)$ with $\deg x=\deg y=1$, $f=\a xy+yx\in S_2$, and $A=S/(f)$.  
Since $H_S(t)=(1+t)/(1-t)$ and $H_{S/(f)}(t)=(1+t)^2$, 
$$\begin{CD} 0 @>>> S(-2) @>f\cdot >> S @>>> S/(f) @>>> 0 \end{CD}$$
is an exact sequence in $\grmod S$ and 
$$\begin{CD} 0 @>>> S(-2) @>\cdot f>> S @>>> S/(f) @>>> 0 \end{CD}$$
is an exact sequence in $\grmod S^o$, so $f\in S_2$ is a regular element. 
Since
\begin{align*}
& xf=x(\a xy+yx)=xyx=(\a xy+yx)\a ^{-1}x=f\a ^{-1}x, \\
& yf=y(\a xy+yx)=\a yxy=(\a xy+yx)\a y=f\a y, 
\end{align*}
$f\in S_2$ is a normal element with the normalizing automorphism $\sf$ given by $\sf(x)=\a^{-1}x, \sf(y)=\a y$.   

For 
$0\neq a, b\in k$, 
$\psi=\{\psi^i\}_{i\in \ZZ}\in \NMF^{\ZZ}_S(f)$ defined by $\Psi^{2i}=\a^{-i}ax+\a ^iby, \Psi^{2i+1}=\a^{-i}b^{-1}x+\a ^{i+1}a^{-1}y$
has the property $\Psi^{i+1}\Psi^i=f$ 
by Example \ref{ex.2.3}, and it is easy to see that $\psi ^{i+2}=\sf(\psi ^i)(-2)$ for every $i\in \ZZ$.  To compute $|\psi|$, we use $\phi=\{\phi^i\}_{i\in \ZZ}\in \NMF^{\ZZ}_S(f)$ defined by $\Phi^{i}=ax+\a ^iby$, which is isomorphic to $\psi$ by Example \ref{ex.2.3}.  Then it is easy to see that $|\psi|=|\phi|=|\a|=|\sf|$.
\end{example} 

\section{Totally Reflexive Modules}

\begin{definition} Let $A$ be a (graded) ring. 
A (graded) right $A$-module $M$ is called \emph{totally reflexive} if
\begin{enumerate}
\item{} $\Ext^i_A(M, A)=0$ for all $i\geq 1$, 
\item{} $\Ext^i_{A^o}(\Hom_A(M, A), A)=0$ for all $i\geq 1$, and 
\item{} the natural biduality map $M\to \Hom_{A^o}(\Hom_A(M, A), A)$ is an isomorphism 
\end{enumerate}
(that is, $\RHom_A(M, A)\cong \Hom_A(M, A)$ and $\RHom_{A^o}(\Hom_A(M, A), A)\cong M$).
The full subcategory of $\mod A$ consisting of totally reflexive modules is denoted by $\TR(A)$.
(The full subcategory of $\grmod A$ consisting of graded totally reflexive modules is denoted by $\TR^{\ZZ}(A)$.) 
\end{definition} 

\begin{remark}
A totally reflexive module is called a module of G-dimension zero or a Gorenstein-projective module in some literature.
\end{remark}

We recall some basic properties of totally reflexive modules.

\begin{lemma} \label{lem.tr} 
Let $A$ be a ring.  
\begin{enumerate}
\item{} If $P\in \mod A$ is 
projective, then $P\in \TR(A)$.  
\item{} If $M$ is a totally reflexive right $A$-module, then $M^*:=\Hom_A(M, A)$ is a totally reflexive left $A$-module.
In particular, if $A$ is left noetherian, then $M\in \TR(A)$ implies $M^*\in \TR(A^o)$. 
\item{}  Let $0 \to N \to P \to M \to 0$ be a short exact sequence of right $A$ modules
where $P$ is projective.
If $M\in \TR(A)$, then $N$ is a totally reflexive right $A$-module.
In particular, if $A$ is right noetherian, then $M\in \TR(A)$ implies $N \in \TR(A)$  (cf. \cite[Lemma 2.3]{AM}). 
\end{enumerate}
\end{lemma} 

\begin{lemma} \label{lem.ev} 
Let $A$ be a right noetherian ring, and $M\in \mod A$.  Suppose that either $\pdim(M)<\infty$ or $\injdim_{A^o}(A)<\infty$.  Then $M\in \TR(A)$ if and only if $\Ext^i_A(M, A)=0$ for all $i\geq 1$.
\end{lemma} 


\begin{remark}
If $A$ is a noetherian AS-Gorenstein algebra, then a finitely generated totally reflexive graded module is the same as a finitely generated maximal Cohen-Macaulay graded module, that is, $\TR^{\ZZ}(A)=\CM^{\ZZ}(A)$.  
\end{remark}

Here is another characterization of a totally reflexive module. 

\begin{definition}
Let $A$ be a (graded) ring.
A \emph{complete resolution} of a finitely generated (graded) right $A$-module $M$ is an infinite exact sequence 
$$\begin{CD} 
\cdots @>d^{i+2}>> P^{i+2} @>d^{i+1}>> P^{i+1} @>d^i>> P^i @>d^{i-1}>> P^{i-1} @>d^{i-2}>>   \cdots \end{CD}$$
where $P^i$ are finitely generated (graded) projective right $A$-modules such that $\Coker d^0\cong M$ and 
$$\begin{CD} 
\cdots @<(d^{i+2})^*<< (P^{i+2})^* @<(d^{i+1})^*<< (P^{i+1})^* @<(d^i)^*<< (P^i)^* @<(d^{i-1})^*<< (P^{i-1})^* @<(d^{i-2})^*<<   \cdots \end{CD}$$
is an exact sequence. If 
$M$ has a complete resolution as above, then
we define $\Omega^{i}M= \Coker d^i$ for every $i\in \ZZ$.
\end{definition}

\begin{lemma} [{\cite[Theorem 3.1]{AM}}] 
\label{lem.comp}
Let $A$ be a noetherian ring.
Then $M\in \TR(A)$ if and only if $M$ has a complete resolution.  
\end{lemma}  

Let $S$ be a ring, $f\in S$, and $A=S/(f)$.  For $\phi\in \NMF_S(f)$, we define $\Coker \phi:=\overline{\Coker \phi^0}\in \mod A$.  


\begin{lemma} \label{lem.cvp}
Let $S$ be a ring, $f\in S$ a regular normal element, and $A=S/(f)$.  For $\phi\in \NMF_S(f)$, $\Coker {\phi^0}\cong (\Coker \phi)_S$ where $(\Coker \phi)_S$ is $\Coker \phi$ viewed as a right $S$-module. 
\end{lemma} 

\begin{proof} By the commutative diagram
\[\xymatrix@R=1.5pc@C=3pc{
F^1 \ar[d] \ar[r]^{\phi^0} &F^0 \ar[d] \ar[r]^(0.4){\e} &\Coker \phi^0 \ar[r] &0\\
(\overline {F^1})_S \ar[d] \ar[r]^{\overline {\phi^0}} &(\overline {F^0})_S \ar[d] \ar[r]^(0.4){\overline {\e}} &(\Coker \phi)_S \ar[r] &0\\
0 &0
}\]
of right $S$-modules, it is standard that $\Coker \phi^0\to (\Coker \phi)_S; \;  \e(a)\mapsto \overline {\e}(\overline a)$ where $a\in F^0$ is a surjective right $S$-module homomorphism. 
For $\e(a)\in \Coker \phi^0$, if $ \overline {\e}(\overline a)=0$, then there exists $b\in F^1$ such that $\overline {\phi^0(b)}=\overline {\phi^0}(\overline b)=\overline a$.
Since $f$ is normal, $(f)=SfS=fS$, so there exists $u\in F^2=S^r$ such that $a-\phi^0(b)=fu$.  It follows that $a-\phi^0(b)\in \Im (f\cdot)=\Im (\phi^0\phi^1)\subset \Im \phi^0$, so $a\in \Im \phi^0$, hence $\e(a)=0$.  
\end{proof} 

\begin{lemma} \label{lem.cok} 
Let $S$ be a ring, $f\in S$ a regular normal element, and $A=S/(f)$.  
If $\phi=\{\phi^i:F^{i+1}\to F^i\}_{i\in \ZZ}\in \NMF_S(f)$,  
then 
$C (\phi)$ is a complete resolution of $\Coker \phi$ in $\mod A$.  
\end{lemma} 

\begin{proof}
If $\overline v\in \Ker (\overline {\phi^i})$ where $v\in F^{i+1}=S^r$, then $\overline {\phi ^i(v)}=\overline {\phi ^i}(\overline v)=\overline 0$.
Since $f$ is normal, there exists $u\in F^{i+2}=S^r$ such that $\phi ^i(v)=fu$.
Since $f$ is regular and  
$$f(v-\phi ^{i+1}(u))=\phi ^{i-1}\phi ^i(v-\phi ^{i+1}(u))=\phi ^{i-1}(\phi ^i(v)-\phi ^i\phi ^{i+1}(u))=\phi ^{i-1}(\phi ^i(v)-fu)=0,$$ 
it follows that $v=\phi ^{i+1}(u)$, so $\overline v\in \Im (\overline {\phi ^{i+1}})$ for every $i\in \ZZ$, hence $C (\phi)$ is an exact sequence.  

In the diagram 
\[\xymatrix@R=1.5pc@C=5.5pc{
\overline {\Hom_S(F^i, S)} \ar[d]_{\cong} \ar[r]^{\overline{\Hom_S(\phi^i, S)}} &\overline{\Hom_S(F^{i+1}, S)} \ar[d]^{\cong} \\
A^r \ar[d]_{\cong} &A^r \ar[d]^{\cong}\\
\Hom_A(\overline {F^i}, A) \ar[r]^{\Hom_A(\overline {\phi ^i}, A)} &\Hom_A(\overline {F^{i+1}}, A),
}\]
both horizontal maps are given by $\cdot \overline {\Phi^i}$, so the diagram commutes, hence we have $\Hom_A(C (\phi), A)\cong C (\Hom_S(\phi, S))$.
By Lemma \ref{lem.dmf}, $\Hom_S(\phi, S)$ is a noncommutative left matrix factorization of $f$ over $S$, so $\Hom_A(C (\phi), A)\cong C (\Hom_S(\phi, S))$ is an exact sequence, hence $C (\phi)$ is a complete resolution of $\Coker \phi$ in $\mod A$. 
\end{proof}

\begin{proposition} \label{q.fun}
Let $S$ be a noetherian ring, $f\in S$ a regular normal element, and $A=S/(f)$.  
\begin{enumerate} 
\item{} $\Coker :\NMF_S(f)\to \TR (A)$
is a functor.   
\item{} $\Omega ^i\Coker \phi=\Coker (\overline {\phi^i}) \in \TR(A)$ for every $i\in \ZZ$.  
\end{enumerate}
\end{proposition} 

\begin{proof} This follows from Lemma \ref{lem.comp} and Lemma \ref{lem.cok}. 
\end{proof} 

\begin{definition} For a (graded) ring $A$ and a (graded) right $A$-module $M$, we define 
$e(M):=\sup\{i\mid \Ext_A^i(M, A)\neq 0\}$.
\end{definition} 
 
\begin{remark} Let $A$ be a ring.  
\begin{enumerate}
\item{} For a right $A$-module $M$, we have $e(M)\in \NN \cup \{-\infty\}$ where $e(M)=-\infty$ if and only if $\Ext_A^i(M, A)=0$ for all $i$.
\item{} If $0\neq M\in \TR(A)$, then $e_A(M)=0$ and $e_{A^o}(M^*)=0$ by definition.  
\item{} Suppose that $A$ is right noetherian and $M\in \mod A$. If either  $\pdim (M) < \infty$ or $\injdim_{A^o}(A)<\infty$,
then $e(M)=-\infty$ if and only if $M=0$, and $e(M)=0$ if and only if $0\neq M\in \TR(A)$ by Lemma \ref{lem.ev}.
\end{enumerate}
\end{remark} 

\begin{lemma} \label{lem.jAS}
Let $A$ be a right noetherian ring and $0\neq M\in \mod A$.
If $\pdim (M)<\infty$ (e.g. if $A$ is regular), then $\pdim (M)=e(M)$. 
In particular, $M\in \mod A$ is a projective module if and only if $\pdim (M)<\infty$ and $M\in \TR(A)$.
\end{lemma} 

\begin{proof} 
Since $M \neq 0$ and $\pdim (M)<\infty$, we clearly have $d:=\pdim (M)\geq e(M)=: e \geq 0$.
Suppose that $d>e$.
Let $\xymatrix@R=1pc@C=1.5pc{
0 \ar[r] &P^{d} \ar[r]^{\delta^{d-1}} &P^{d-1} \ar[r]^(0.6){\delta^{d-2}} &\cdots \ar[r] &P^1 \ar[r]^{\delta^0} &P^0 \ar[r]^{\e} &M \ar[r] &0}$ be a projective resolution of $M$ in $\mod A$.  If $K:=\Coker \delta^{d-1}\in \mod A$, then $\pdim (K)=1$. 
Since $P^{d}$ is finitely generated projective and $d>e$, 
$\Ext_A^1(K, P^{d})\cong \Ext_A^d(M, P^{d})=0$, so the exact sequence $0\to P^{d}\to P^{d-1}\to K\to 0$ splits.
It follows that $K$ is projective, which is a contradiction, hence $d=e$. 
\end{proof} 

\begin{definition} 
Let $S$ be a (graded) ring, $f\in S$ a (homogeneous) element, and $A=S/(f)$. We define 
\begin{align*}
&\TR_S(A):=\{M\in \TR(A)\mid \pdim _S(M)<\infty\}\\
(&\TR_S^{\ZZ}(A):=\{M\in \TR^{\ZZ}(A)\mid \pdim _S(M)<\infty\}).
\end{align*}
\end{definition} 
Note that if $S$ is a (graded) regular ring, then $\TR_S(A)=\TR(A)$ ($\TR_S^{\ZZ}(A)=\TR^{\ZZ}(A)$).

\begin{lemma} \label{lem.pd} 
Let $S$ be a right noetherian ring, $f\in S$ a regular normal element, and $A=S/(f)$.
For $0\neq M\in \mod A$, if $\pdim_S(M)<\infty$ (e.g. if $S$ is regular), then $\pdim_S(M)=e_A(M)+1$.  
In particular, if $0\neq M\in \TR_S(A)$,  
then $\pdim_S(M)=1$. 
\end{lemma} 

\begin{proof} An exact sequence 
\[
\xymatrix@R=1.5pc@C=2pc{
0 \ar[r] &{_{\sf}}S \ar[r]^{f\cdot} &S \ar[r] &A \ar[r] &0 
}
\]
in $\mod S^e$ induces the following commutative diagram 
\[
\xymatrix@R=1.5pc@C=2pc{
0 \ar[r] &\Hom_S(A,S) \ar[r] &\Hom_S(S,S) \ar_{\cong}[d] \ar[rr]^{\Hom_S(f\cdot,S)} &&\Hom_S({_{\sf}}S,S) \ar^{\cong}[d] \ar[r] &\Ext^1_S(A,S) \ar[r] &0\\
&0 \ar[r] &S \ar[rr]^{\cdot f} &&S_{\sf} \ar[r]& A_{\sf} \ar[r] &0
}
\]
in $\mod S^e$, so $\Hom_S(A, S)=0$ and $\Ext^1_S(A, S)\cong A_{\sf}$ as $S$-$S$ bimodules. Note that $\Ext^1_S(A, S)$ has a right $A$-module structure induced by the left action on $A$, which recovers the original right $S$-module structure on $\Ext^1_S(A, S)$ via the map $S\to A$.
On the other hand, since $\nu(f)=f$, we see that $\nu$ induces $\overline {\nu}\in \Aut A$, so $A_{\nu}$ has a right $A$-module structure by the identification $A_{\sf}=A_{\overline {\nu}}$, which recovers the original right $S$-module structure on $A_{\nu}$ via the map $S\to A$. Hence $\Ext^1_S(A, S)\cong A_{\sf}=A_{\overline {\nu}}\cong A$ as right $A$-modules.
Moreover, since $\pdim _S(A)=1$, we have $\Ext_A^i(A, S)=0$ for every $i\geq 2$, so
\begin{align*}
\RHom_S(M, S) & \cong\RHom_S(M\lotimes _AA, S)\cong \RHom_A(M, \RHom_S(A, S)) \\
& \cong \RHom_A(M, A[-1])\cong \RHom_A(M, A)[-1].
\end{align*}
Since $S$ is right noetherian, $\pdim_S(M)<\infty$ and $e_S(M)\geq 0$, we obtain $\pdim_S(M)=e_S(M)=e_A(M)+1$ by Lemma \ref{lem.jAS}. 
\end{proof}

\begin{remark}
With obvious changes, all statements in this section also hold true in the graded case.
\end{remark}

\section{Factor Categories} 

In this section, we generalize the graded version of Eisenbud's matrix factorization theorem (Theorem \ref{thm.mot} (2), (3)) to noncommutative, non-hypersurface algebras.

Let $S$ be a graded quotient algebra of a right noetherian connected graded regular algebra and $M\in \grmod S$.
By \cite[Theorem 2.4]{SZ}, $H_M(t)$ is a rational function and $\GKdim M\in \NN$ is finite, which is given by the order of the pole of $H_M(t)$ at $t=1$.
It follows that $H_M(t)=q_M(t)/(1-t)^{\GKdim M}$ where $q_M(t)$ is a rational function such that $0\neq q_M(1)\in \QQ$ is finite.
We define 
$$\hrank M=\lim _{t\to 1}H_M(t)/H_S(t).$$ 
If $F=\bigoplus _{i=1}^rS(-m_i)$ is a graded free right $S$-module of rank $r$, then $H_F(t)=\sum _{i=1}^rt^{m_i}H_S(t)$, so $\hrank F=r=\rank F$.
On the other hand, since $H_M(t)/H_S(t)=(1-t)^{\GKdim S-\GKdim M}q_M(t)/q_S(t)$, we see that $\hrank M=0$ if and only if $\GKdim M<\GKdim S$.
If $0\to L\to M\to N\to 0$ is an exact sequence in $\grmod S$, then $H_M(t)=H_L(t)+H_N(t)$, so $\hrank M=\hrank L+\hrank N$.
(Thus the problem of Remark \ref{rem.tmf} can be avoided.)

\begin{theorem} \label{thm.m1}
Let $S$ be a graded quotient algebra of a right noetherian connected graded regular algebra, $f\in S_d$ a regular normal element, and $A=S/(f)$. 
Then the image of the functor $\Coker:\NMF^{\ZZ}_S(f)\to \TR^{\ZZ}(A)$ is $\TR^{\ZZ}_S(A)$.  
\end{theorem} 

\begin{proof}
For $\phi\in \NMF_S^{\ZZ}(f)$, $\phi^0:F^1\to F^0$ is injective, so 
$$\begin{CD} 0 @>>> F^1 @>\phi^0>> F^0 @>>> \Coker \phi^0 @>>> 0\end{CD}$$
is a free resolution of $\Coker \phi^0$ in $\grmod S$.  Since $\Coker \phi^0$ is the same as $\Coker \phi:=\overline {\Coker \phi^0}\in \grmod A$ viewed as a graded right $S$-module by Lemma \ref{lem.cvp}, $\pdim_S(\Coker \phi)\leq 1$, so $\Coker \phi\in \TR_S^{\ZZ}(A)$.  (Note that if $\Coker \phi\neq 0$, then $\pdim_S(\Coker \phi)=1$.)  

Conversely, if $0\neq M\in \TR^{\ZZ}_S(A)$, 
then there exists a graded free right $S$-module $F$ of finite rank such that $\overline{\e}:\overline {F}\to M\to 0$ is an exact sequence in $\grmod A$.  Since $\pdim_S(M)=1$ by Lemma \ref{lem.pd}, there exists a graded free right $S$-module $G$ of finite rank such that 
$$\begin{CD} 0 @>>> G @>\phi>> F @>\e>> M @>>> 0\end{CD}$$
is a free resolution of $M$ in $\grmod S$.  
Since $S$ is a graded quotient algebra of a right noetherian connected graded regular algebra, $H_S(t)$ is a rational function and $H_A(t)=(1-t^d)H_S(t)$, so $\GKdim M\leq \GKdim A=\GKdim S-1<\infty$.
It follows that $\hrank_S M=0$,
so $r:=\rank F = \hrank_S F =\hrank_S G =\rank G$.  Since $\e:F\to \overline {F}\to M$, for every $u\in F$, $fu\in \Ker \e=\Im \phi$.  Since $\phi$ is injective, there exists a unique $\psi (u)\in G$ such that $\phi(\psi(u))=fu$.  For $u, v\in F$ and $a, b\in S$, 
$$\phi(\psi(ua+vb)) =f(ua+vb)=fua+fvb=\phi(\psi(u))a+\phi(\psi(v))b=\phi(\psi(u)a+\psi(v)b),$$
so $\psi:F(-d)\to G$ is a graded right $S$-module homomorphism such that $\phi\psi=f\cdot:F(-d)\to F$.  By Theorem \ref{thm.nmf} (2), there exists a unique $\phi\in \NMF^{\ZZ}_S(f)$ such that $\phi^0=\phi$ and $\phi^1=\psi$ so that $\Coker \phi =\overline {\Coker \phi^0}\cong M$.  
\end{proof} 

\begin{proposition} \label{prop.mfr} 
Let $S$ be a graded quotient algebra of a right noetherian connected graded regular algebra, $f\in S_d$ a regular normal element, and $A=S/(f)$. 
If $M\in \TR^{\ZZ}_S(A)$ has no free summand, then there exists 
$\phi
\in \NMF_S^{\ZZ}(f)$ such that $C (\phi)^{\geq 0}$ is a minimal free resolution of $M$.
\end{proposition}

\begin{proof}
By Theorem \ref{thm.m1}, there exists $\phi\in \NMF^{\ZZ}_S(f)$ such that $\Coker \phi\cong M$.  By Proposition \ref{prop.tmf}, $(\phi^0, \phi^1\nu)\in \TMF_S^{\ZZ}(f)$ and $C'(\phi^0, \phi^1\nu)\cong C(\phi)$. Since $M$ has no free summand, $C(\phi)^{\geq 0}\cong C' (\phi^0, \phi^1\nu)^{\geq 0}$ is a minimal free resolution of $M$ by \cite[Proposition 2.9]{CCKM}.
\end{proof} 

Let $\cC$ be an additive category and $\cF$ a set of objects of $\cC$ closed under direct sums.
Then the \emph{factor category} $\cC/\cF$ has $\Obj(\cC/\cF) =\Obj(\cC)$ and $\Hom_{\cC/\cF}(M, N) = \Hom_{\cC}(M, N)/\cF(M,N)$ for $M, N\in \Obj(\cC/\cF)=\Obj(\cC)$,
where $\cF(M,N)$ is the subgroup consisting of all morphisms from $M$ to $N$ that factor through objects in $\cF$. Note that $\cC/\cF$ is also an additive category.

\begin{definition} \label{def.trivnmf}
Let $S$ be a ring and $f\in S$.  For a free right $S$-module $F$, we define $\phi_F, {_F\phi}\in \NMF_S(f)$ by
$$\begin{array}{lll}
& \phi_F^{2i}=\id_F:F\to F, & \phi_F^{2i+1}=f\cdot : F\to F, \\
& {_F\phi}^{2i}=f\cdot:F\to F, & {_F\phi}^{2i+1}=\id _F: F\to F.
\end{array}$$
We define
$\cF :=\{\phi_F\mid F\in \mod S \; \textnormal{is free} \}$,
$\cG :=\{\phi_F\oplus {_G\phi} \mid F, G\in \mod S \; \textnormal{are free} \}$, and
$\uNMF_S(f):=\NMF_S(f)/\cG$.

Let $S$ be a graded ring and $f\in S_d$.  For a graded free right $S$-module $F$, 
we define $\phi_F, {_F\phi}\in \NMF_S^{\ZZ}(f)$ by
$$\begin{array}{lll}
& \phi_F^{2i}=\id_F:F(-id)\to F(-id), & \phi_F^{2i+1}=f\cdot : F(-id-d)\to F(-id), \\
& {_F\phi}^{2i}=f\cdot: F (-id-d)\to F(-id), & {_F\phi}^{2i+1}=\id _F: F (-id-d)\to F(-id-d).
\end{array}$$
We define
$\cF :=\{\phi_F\mid F\in \grmod S \; \textnormal{is free} \}$,
$\cG :=\{\phi_F\oplus {_G\phi} \mid F, G\in \grmod S \; \textnormal{are free} \}$, and
$\uNMF_S^{\ZZ}(f):=\NMF_S^{\ZZ}(f)/\cG$.
\end{definition} 

\begin{proposition} 
\label{prop.m2}
Let $S$ be a ring, $f\in S$ and $A=S/(f)$.  Then the functor $\Coker :\NMF_S(f)\to \mod A$ induces a fully faithful functor $\NMF_S(f)/\cF\to \mod A$.
A similar result holds in the graded case.
\end{proposition} 

\begin{proof} 
Since $\Coker \phi_F=0$, the functor $\Coker :\NMF_S(f)\to \mod A$ induces a functor $\NMF_S(f)/\cF\to \mod A$. 
For any $\phi, \psi\in \NMF_S(f)$, it is enough to show that $\Coker :\Hom_{\NMF_S(f)}(\phi, \psi)\to \Hom_A(\Coker \phi, \Coker \psi)$ induces an isomorphism 
$$\Hom_{\NMF_S(f)}(\phi, \psi)/\cF(\phi, \psi)\to \Hom_A(\Coker \phi, \Coker \psi).$$

Every $\a\in \Hom_A(\Coker \phi, \Coker \psi)$ can be viewed as a right $S$-module homomorphism $\a:\Coker \phi^0\to \Coker \psi^0$, which extends to a commutative diagram 
\[\xymatrix@R=2pc@C=2.5pc{
0 \ar[r] &F^{1} \ar[d]_{\mu^1} \ar[r]^{\phi^0} &F^0 \ar[r] \ar[d]^{\mu^0} &\Coker \phi \ar[r] \ar[d]^{\a}& 0 \\
0 \ar[r] &G^{1} \ar[r]^{\psi^0} &G^0 \ar[r] &\Coker \psi \ar[r] & 0. 
}\]
By Theorem \ref{thm.nmf} (4), there exists $\mu\in \Hom_{\NMF_S(f)}(\phi, \psi)$ such that $\Coker \mu=\a$. 

If $\mu \in \cF(\phi, \psi)$ so that it factors through $\mu :\phi \to \phi_F\to \psi$ for some free right $S$-module $F$, then we have the following commutative diagram
\[\xymatrix@R=1.5pc@C=2.5pc{
0 \ar[r] &F^{1} \ar[d] \ar[r]^{\phi^0} &F^0 \ar[r] \ar[d] &\Coker \phi \ar[r] \ar[d]& 0 \\
0 \ar[r] &F \ar[d] \ar[r]^{\id_F} &F \ar[r] \ar[d] &\Coker \id_F = 0 \ar[r] \ar[d]& 0 \\
0 \ar[r] &G^{1} \ar[r]^{\psi^0} &G^0 \ar[r] &\Coker \psi \ar[r] & 0, 
}\]
so $\Coker \mu=0$.  

Conversely, if $\Coker \mu=0$ so that
\[\xymatrix@R=2pc@C=2.5pc{
0 \ar[r] &F^{1} \ar[d]_{\mu^1} \ar[r]^{\phi^0} &F^0 \ar[r] \ar[d]^{\mu^0} &\Coker \phi \ar[r] \ar[d]^{0}& 0 \\
0 \ar[r] &G^{1} \ar[r]^{\psi^0} &G^0 \ar[r] &\Coker \psi \ar[r] & 0,
}\]
then there exists a right $S$-module homomorphism $\mu':F^0\to G^1$ such that $\psi^0\mu'=\mu^0$.
Since we have a commutative diagram
\[\xymatrix@R=2pc@C=2.5pc{
0 \ar[r] &F^{1} \ar[d]_{\phi^0} \ar[r]^{\phi^0} &F^0 \ar[r] \ar[d]^{\id_{F^0}} &\Coker \phi \ar[r] \ar[d]& 0 \\
0 \ar[r] &F^0 \ar[d]_{\mu'} \ar[r]^{\id_{F^0}} &F^0 \ar[r] \ar[d]^{\mu^0} &\Coker \id_{F^0} = 0 \ar[r] \ar[d]& 0 \\
0 \ar[r] &G^{1} \ar[r]^{\psi^0} &G^0 \ar[r] &\Coker \psi \ar[r] & 0, 
}\]
there exist morphisms $\phi\to \phi_F$ and $\phi_F\to \psi$ by Theorem \ref{thm.nmf} (4) (existence).  Since $\psi^0\mu'\phi^0=\mu^0\phi^0=\psi^0\mu^1$ and $\psi^0$ is injective, $\mu'\phi^0=\mu^1$, so $\mu$ factors through $\mu :\phi \to \phi_{F^0}\to \psi$ by Theorem \ref{thm.nmf} (4) (uniqueness), hence the functor $\NMF_S(f)/\cF\to \mod A$ is fully faithful. 
\end{proof} 

The following two theorems are noncommutative graded versions of Eisenbud's matrix factorization theorem (Theorem \ref{thm.mot} (2), (3)).  

\begin{theorem} \label{thm.m3} 
Let $S$ be a graded quotient algebra of a right noetherian connected graded regular algebra, $f\in S_d$ a regular normal element, and $A=S/(f)$. 
Then the functor $\Coker:\NMF^{\ZZ}_S(f)\to \TR^{\ZZ}(A)$ induces an equivalence functor $\NMF_S^{\ZZ}(f)/\cF\to \TR^{\ZZ}_S(A)$.  
In particular, if $S$ is a noetherian AS-regular algebra, then 
$\NMF^{\ZZ}_S(f)/\cF\cong \TR_S^{\ZZ}(A)=\TR^{\ZZ}(A)=\CM^{\ZZ}(A)$.
\end{theorem} 

\begin{proof}
It follows from Theorem \ref{thm.m1} and Proposition \ref{prop.m2}. 
\end{proof} 

Let $S$ be a graded algebra, $f\in S_d$ a regular normal element and $A=S/(f)$.
If $0\neq P\in \grmod A$ is free, then as we have seen that $P\in \TR^{\ZZ}(A)$ and $\pdim _S(P)=1$, so $P\in \TR_S^{\ZZ}(A)$.
We define 
$\cP:=\{P\in \grmod A\mid P \; \textnormal{is free}\}$, and $\uTR_S^{\ZZ}(A):=\TR_S^{\ZZ}(A)/\cP$.

The following theorem is analogous to \cite [Theorem 5.8]{CCKM} in the case of noncommutative graded hypersurfaces.

\begin{theorem}  \label{thm.m4} 
If $S$ is a graded quotient algebra of a right noetherian connected graded regular algebra, $f\in S_d$ is a regular normal element, and $A=S/(f)$, 
then the functor $\Coker:\NMF^{\ZZ}_S(f)\to \TR^{\ZZ}(A)$ induces an equivalence functor $\underline {\Coker}:\uNMF^{\ZZ}_S(f)\to \underline {\TR}_S^{\ZZ}(A)$. 
In particular, if $S$ is a noetherian AS-regular algebra, then 
$\uNMF^{\ZZ}_S(f)\cong \underline {\TR}_S^{\ZZ}(A)
=\uTR^{\ZZ}(A)= \uCM^{\ZZ}(A)$.
\end{theorem}   

\begin{proof} For every $\phi_F\oplus {_G\phi}\in \cG$, $\Coker (\phi_F\oplus{_G\phi})=\overline G\in \cP$. 
On the other hand, suppose that $\Coker \phi=\overline F\in \cP$ where $F\in \grmod S$ is free. 
Since 
$$\begin{CD} 0 @>>> F(-d) @>f\cdot >> F @>>> \overline F @>>>0\end{CD}$$
is the minimal free resolution of $\overline F$ in $\grmod S$, we have a commutative diagram
\[\xymatrix@R=1.5pc@C=2.5pc{
&0 \ar[d]&0 \ar[d] &0 \ar[d]\\
0 \ar[r] &F(-d) \ar[d] \ar[r]^{f\cdot} &F \ar[r] \ar[d] &\overline F  \ar[r] \ar[d]& 0 \\
0 \ar[r] &F^1 \ar[d] \ar[r]^{\phi^0} &F^0 \ar[r] \ar[d] &\Coker \phi \ar[r] \ar[d]& 0 \\
0 \ar[r] &G^1 \ar[d] \ar[r]^{\psi^0} &G^0 \ar[r] \ar[d] &0 \ar[r] \ar[d]& 0\\
&0 &0 &0
}\]
where vertical sequences are split exact.  Since $\psi^0:G^1\to G^0$ is an isomorphism, $\phi\cong {_F\phi}\oplus \phi_{G^0}\in \cG$.  It follows that the equivalence functor $\NMF_S^{\ZZ}(f)/\cF\to \TR^{\ZZ}_S(A)$ of Theorem \ref{thm.m3} restricts to a bijection from  $\cG/\cF$ to $\cP$, so 
it induces an equivalence functor $\underline {\Coker}:\uNMF^{\ZZ}_S(f)\to \underline {\TR}_S^{\ZZ}(A)$.  
\end{proof} 

\begin{remark}
For $\phi\in \NMF_S(f)$, we define $\phi[1]\in \NMF_S(f)$ by $\phi[1]^i=\phi^{i+1}$ for every $i\in \ZZ$.
Note that $\phi_F[1]\cong {_F\phi}$ and ${_F\phi}[1]\cong \phi_F$.  
In the above theorem, if $S$ is a noetherian AS-regular algebra, then $\uTR^{\ZZ}(A)=\uCM^{\ZZ}(A)$ has a structure of a triangulated category with the translation functor $\Omega ^{-1}$,
so $\uNMF_S^{\ZZ}(f)$ has a structure of a triangulated category with the translation functor $[-1]$ (i.e., $\{\phi^i\}_{i \in \ZZ}\mapsto \{\phi^{i-1}\}_{i \in \ZZ}$). 
\end{remark}

\section{An Application to Skew Exterior Algebras} 

In this last section, we will apply our theory of noncommutative matrix factorizations to skew exterior algebras,
which are hardly noncommutative hypersurfaces.
An application to noncommutative quadric hypersurfaces will be discussed in our subsequent paper \cite{MU2}.  
Throughout this section, {\bf we assume that every graded algebra is finitely generated in degree 1 over $k$.}  

\subsection{Co-point Modules} 

\begin{definition} 
Let $A$ be a graded algebra and $M\in \grmod A$.  
\begin{enumerate}
\item{} We say that $M$ is a \emph{point module} over $A$ if $M=M_0A$ and $H_M(t)=1/(1-t)$.  
\item{} We say that $M$ is a \emph{co-point module} over $A$ if $M$ has a linear free resolution of the form 
\[ \cdots \to A(-3) \to A(-2)\to A(-1)\to A\to M\to 0.\] 
\end{enumerate}
\end{definition} 

\begin{definition}
Let $A$ be a graded algebra.  We say that $M\in \grmod A$ is an \emph{$r$-extension of (shifted) point modules} if $M$ has a filtration 
$$M=M_0\supset M_1\supset \cdots \supset M_r=0$$
such that $M_i/M_{i+1}\in \grmod A$ is a (shifted) point module for every $i=0, \dots, r-1$. 

We can also similarly define an \emph{$r$-extension of (shifted) co-point modules}.
\end{definition}

Write $A=k\<x_1, \dots, x_n\>/I$ where $I$ is a homogeneous ideal of $k\<x_1, \dots, x_n\>$ such that $I_0=I_1=0$ (and $\deg x_i=1$ for all $i=1, \dots, n$ as always assumed).
For a point $p=(a_1, \dots, a_n)\in \PP^{n-1}$, we define $N_p:=A/(a_1x_1+\cdots +a_nx_n)A\in \grmod A$.
Note that, for $p, q\in \PP^{n-1}$, $N_p\cong N_q$ if and only if $p=q$.  If $N\in \grmod A$ is a co-point module, then $N\cong N_p$ for some $p\in \PP^{n-1}$.
Let $E:=\{p\in \PP^{n-1}\mid N_p \textnormal {\; is a co-point module over $A$}\}$.  If $N$ is a co-point module, then the first syzygy module $\Omega N(1)$ is also a co-point module, so there is a map $\t:E\to E$ defined by $\Omega N_p(1)\cong N_{\t(p)}$.
The pair $\cP^!(A):=(E, \t)$ is called the \emph{co-geometric pair} of $A$ (see \cite {Mck}).

We denote by $\lin A$ the full subcategory of $\grmod A$ consisting of modules having linear free resolutions.  
We say that $A$ is \emph{Koszul} if $k = A/A_{\geq 1} \in \lin A$.
For $M\in \grmod A$, we define $E(M):=\bigoplus _{i\in \NN}\Ext^i_A(M, k)$.  If $A$ is a Koszul algebra, then $A^!:=E(k)^o=\bigoplus _{i\in \NN}\Ext^i_A(k, k)^o$ is also a Koszul algebra, and $E$ induces a duality functor $E:\lin A\to \lin A^!$.  

\begin{lemma}[{\cite [Theorem 3.4]{Mck}}]
If $A$ is a Koszul algebra, then the duality functor $E:\lin A\to \lin A^!$ induces a bijection from the set of isomorphism classes of co-point modules over $A$ to the set of isomorphism classes of point modules having linear free resolutions over $A^!$.   
\end{lemma} 

For a co-point module $N_p\in \lin A$ where $p\in E$, we denote by $M_p:=E(N_p)\in \lin A^!$ the corresponding point module.  

\begin{definition} A noetherian $n$-dimensional AS-regular algebra $A$ is called a \emph{quantum polynomial algebra} if 
\begin{enumerate}
\item{} $H_A(t)=1/(1-t)^n$, and  
\item{} $j(M)+\GKdim M=\GKdim A \;(=n)$ for every $M\in \grmod A$ where $j(M):=\inf\{i\in \NN\mid \Ext^i_A(M, A)\neq 0\}$. 
\end{enumerate}
\end{definition} 

\begin{example} \label{ex.qpa} 
\begin{enumerate}
\item{} A (commutative) polynomial algebra is obviously a quantum polynomial algebra.  
\item{} Every noetherian 3-dimensional quadratic AS-regular algebra is a quantum polynomial algebra
(see \cite[Corollary 6.2]{L}). 
\item {} The skew polynomial algebra $T=k\<x_1, \dots, x_n\>/(\a_{ij}x_ix_j+x_jx_i)_{1\leq i,  j\leq n, i\neq j}$ where $\a_{ij}\in k$ such that $\a_{ij}\a_{ji}=1$ for every $1\leq i, j\leq n, i\neq j$ 
is a quantum polynomial algebra by \cite[Lemma (\rnum{2}) on page 184]{LS}.  
\end{enumerate}
\end{example} 

\begin{remark} Every co-point module has a linear free resolution by definition, but not every point module has a linear free resolution even if $A$ is Koszul.  However, if $A^!$ is a quantum polynomial algebra,   
then $A$ and $A^!$ are Koszul, and every point module over $A^!$ has a linear free resolution by \cite [Corollary 5.7]{Mck}, so the duality functor $E:\lin A\to \lin A^!$ induces a bijection from the set of isomorphism classes of co-point modules over $A$ to the set of isomorphism classes of point modules over $A^!$.  
\end{remark} 

\begin{lemma} \label{lem.liex} 
The following categories are closed under extensions. 
\begin{enumerate}
\item{} $\lin A$ for a 
graded algebra $A$.  
\item{} $\TR^{\ZZ}(A)$ for a 
graded algebra $A$. 
\item{} $\TR^{\ZZ}_S(A)$ for a graded algebra $S$, $f\in S_d$, and $A=S/(f)$.
\end{enumerate}
\end{lemma} 

\begin{proof} Left to the reader. 
\end{proof} 

Let $S$ be a graded algebra, $f\in S_d$ a regular normal element, and $A=S/(f)$.   We define 
\begin{align*}
 \NMF_S^0(f)&:= \{\{\phi^i:F^{i+1}\to F^i\}_{i\in \ZZ}\in \NMF^{\ZZ}_S(f)\mid F^0\;\textnormal{is generated in degree 0} \; (\textnormal{i.e.,}\; F^0\cong S^r )\}, \\
 \TR^0(A)&:=\{M\in \TR^{\ZZ}(A)\mid M=M_0A\}, \\
 \TR^0_S(A)&:=\{M\in \TR^{\ZZ}_S(A)\mid M=M_0A\}.
\end{align*}  

\begin{proposition} \label{prop.lin} 
Let $S$ be a graded quotient algebra of a right noetherian connected graded regular algebra, $f\in S_2$ a regular normal element, and $A=S/(f)$. 
\begin{enumerate}
\item{} $\TR^0_S(A)=\TR^{\ZZ}_S(A)\cap \lin A$.  In particular, $\TR^0_S(A)$ is closed under extensions.  
\item{} If $\phi\in \NMF^0_S(f)$ is of rank 1 such that $\Coker \phi$ has no free summand, then $\Coker \phi$ is a co-point module. 
\end{enumerate}
\end{proposition} 

\begin{proof}
(1) If $M\in \TR^0_S(A)$ is free, then $M\in \TR^{\ZZ}_S(A)\cap \lin A$.  If $M\in \TR^0_S(A)$ has no free summand, then there exists $\phi=\{\phi^i:F^{i+1}\to F^i\}_{i\in \ZZ}\in \NMF^{\ZZ}_S(f)$ such that $C (\phi)^{\geq 0}$ is the minimal free resolution of $M$ by Proposition \ref{prop.mfr}.
Since $\deg f=2$, we have $F^{i+2}\cong F^i(-2)$ for every $i\in \ZZ$.
Since $M=M_0A$, we see $\overline {F^0}\cong A^r$, so $C (\phi)^{\geq 0}$ is a linear free resolution of $M$, hence $M\in \TR^{\ZZ}_S(A)\cap \lin A$.  The converse is clear. 

(2) If $\phi=\{\phi^i:F^{i+1}\to F^i\}_{i\in \ZZ}\in \NMF^0_S(f)$ is of rank 1 such that 
$\Coker \phi$ has no free summand, then $C(\phi)^{\geq 0}$ is the linear free resolution of $\Coker \phi\in \TR^0_S(A)$ by (1).  
Since $\overline {F^i}\cong A(-i)$,
$\Coker \phi$ is a co-point module.   
\end{proof}

\begin{example}
Since the skew polynomial algebra $T=k\<x_1, \dots, x_n\>/(\a_{ij}x_ix_j+x_jx_i)$
is a quantum polynomial algebra, 
$S:=T/(x_1^2, \dots, x_{n-1}^2)$ is a graded quotient algebra of a noetherian AS-regular algebra.
In this case, $f=x_n^2\in S_2$ is a regular normal element, $A:=S/(f)$ is a skew exterior algebra, and $A^!\cong k\<x_1, \dots, x_n\>/(x_ix_j-\a_{ij}x_jx_i)$ is again a skew polynomial algebra. 
\end{example} 

\begin{theorem} \label{thm.copo}
Let $T=k\<x_1, \dots, x_n\>/(\a_{ij}x_ix_j+x_jx_i)$ be a skew polynomial algebra, and put $S:=T/(x_1^2, \dots, x_{n-1}^2)$.
Let $f=x_n^2\in S_2$ so that $A=S/(f)$ is a skew exterior algebra, and let $\cP^!(A)=(E, \t)$.  
\begin{enumerate}
\item{} For $p\in E$, $N_p\in \TR^0_S(A)$ if and only if $p\not\in \cV(x_n)$.  
\item{} If $M\in \grmod A$ is an $r$-extension of co-point modules $N_{p_i}$ for $p_i\in E\setminus \cV(x_n)$, then there exists $\phi\in \NMF^0_S(f)$ of rank $r$ such that $M\cong \Coker \phi$.  
\end{enumerate}
\end{theorem} 

\begin{proof} (1) Recall that $\ell_{ij}=\cV(x_1, \dots, x_{i-1}, x_{i+1}, \dots, x_{j-1}, x_{j+1}, \dots, x_n)
\subset E$ and $\t(\ell_{ij})=\ell_{ij}$ for every $ 1\leq i<j\leq n$ by \cite[Theorem 4.1]{Ue}, so
 $E\setminus \cV(x_n)\neq \emptyset$.  For $p\in E$ and $i\in \ZZ$, let $\t^i(p)=(a_{i1}, \dots, a_{in})\in E\subset \PP^{n-1}$.  Then $p\in E\setminus \cV(x_n)$ if and only if $\t^i(p)\in E\setminus \cV(x_n)$ for every $i\in \ZZ$ if and only if $a_{in}\neq 0$ for every $i\in \ZZ$.  Since 
$$\begin{CD} \cdots @>{(\sum_{j=1}^na_{2j}x_j)\cdot }>>A(-2) @>{(\sum_{j=1}^na_{1j}x_j)\cdot }>> A(-1) @>{(\sum_{j=1}^na_{0j}x_j)\cdot }>> A @>>> N_p @>>> 0\end{CD}$$
is a linear free resolution of $N_p$ over $A$, we have
$(\sum_{j=1}^na_{i+1, j}x_j)(\sum_{j=1}^na_{ij}x_j)=0$ in $A$.  It follows that $(\sum_{j=1}^na_{i+1, j}x_j)(\sum_{j=1}^na_{ij}x_j)$ is a linear combination of $\{\a_{ij}x_ix_j+x_jx_i\}_{1\leq i<j\leq n}\cup \{x_i^2\}_{1\leq i\leq n}$ in $k\<x_1, \dots, x_n\>$, so $(\sum_{j=1}^na_{i+1, j}x_j)(\sum_{j=1}^na_{ij}x_j)=a_{i+1, n}a_{in}x_n^2=a_{i+1, n}a_{in}f$ in $S$ for every $i\in \ZZ$.  Let $\Phi^i:=\sum_{j=1}^na_{ij}x_j$ for $i\in \ZZ$.  By the uniqueness of the linear free resolution, $N_p\in \TR^0_S(A)$ if and only if $\Phi$ induces $\phi\in \NMF^0_S(f)$ such that $\Coker \phi\cong N_p$
if and only if $a_{in}\neq 0$ for every $i\in \ZZ$ if and only if $p\not \in \cV(x_n)$.  

(2) Since $N_{p_i}\in \TR^0_S(A)$ for every $p_i\in E\setminus \cV(x_n)$ by (1), $M\in \TR^0_S(f)$ by Proposition \ref{prop.lin} (1).  
\end{proof} 

\subsection{Indecomposable Noncommutative Matrix Factorizations} 

Let $A$ be a right noetherian graded algebra.
We denote by $\tors A$ the full subcategory of $\grmod A$ consisting of finite dimensional modules over $k$, and by $\tails A:=\grmod A/\tors A$ the Serre quotient category.  Note that $\Obj(\tails A)=\Obj(\grmod A)$, and, for $M, N\in \Obj(\tails A)=\Obj(\grmod A)$, $M\cong N$ in $\tails A$ if and only if $M_{\geq n}\cong N_{\geq n}$ in $\grmod A$ for some $n\in \ZZ$.  
We call $\tails A$ the \emph{noncommutative projective scheme} associated to $A$ since if $A$ is commutative (and finitely generated in degree 1 as always assumed), then $\tails A$ is equivalent to the category of coherent sheaves on $\Proj A$.
Let $\Tors A$ be the full subcategory of $\GrMod A$ consisting of direct limits of modules in $\tors A$, and $\Tails A:=\GrMod A/\Tors A$.
It is known that the quotient functor $\pi :\GrMod A\to \Tails A$ is exact and has a right adjoint $\omega :\Tails A\to \GrMod A$. See \cite{AZ} for details.

\begin{proposition} \label{prop.fpm}
Let $A$ be a quantum polynomial algebra. 
If $M\in \grmod A$ is an $r$-extension of shifted point modules, then the following hold: 
\begin{enumerate}
\item{} $\R^1\omega\pi M=0$. 
\item{} $\dim _k(\omega \pi M)_n=r$ for every $n\in \ZZ$. 
\item{} $(\omega \pi M)(n)_{\geq 0}\in \lin A$ for every $n\in \ZZ$.
\end{enumerate}
\end{proposition} 

\begin{proof}
By induction on $r$.  For the case $r=1$, it is known that if $M\in \grmod A$ is a shifted point module, then $\R^1\omega \pi M=\H^2_{\fm}(M)=0$.
By \cite [Corollary 5.6]{Mck}, $M\in \lin A$, so (2) and (3) follows from \cite [Lemma 6.5, Proposition 6.6]{Mcfk}. 

For $r\geq 1$, if $M\in \grmod A$ has a filtration 
$$M=M_0\supset M_{1}\supset \cdots \supset M_r=0$$
such that $M_i/M_{i+1}\in \grmod A$ is a shifted point module for every $i=0, \dots, r-1$,  
then $M_1$ is an $(r-1)$-extension of shifted point modules, so 
we have $\R^1\omega \pi M_1=0$, $\dim _k(\omega \pi M_1)_n=r-1$, and $(\omega \pi M_1)(n)_{\geq 0}\in \lin A$ for every $n\in \ZZ$ by induction.
An exact sequence $0\to M_1\to M\to M/M_1\to 0$ induces an exact sequence 
$$0\to \omega \pi M_1\to \omega \pi M\to \omega \pi (M/M_1)\to \R^1\omega \pi M_1\to \R^1\omega \pi M\to \R^1\omega \pi (M/M_1).$$
Since $M/M_1\in \grmod A$ is a shifted point module, $\R^1\omega \pi (M/M_1)=0$, so $\R^1\omega \pi M=0$ and we have an exact sequence
$$0\to (\omega \pi M_1)(n)_{\geq 0}\to (\omega \pi M)(n)_{\geq 0}\to (\omega \pi (M/M_1))(n)_{\geq 0}\to 0.$$
Since $\dim _k(\omega \pi M_1)_n=r-1$ and $\dim _k(\omega \pi (M/M_1))_n=1$, we have $\dim _k(\omega \pi M)_n=r$ for every $n\in \ZZ$.  
Since $(\omega \pi M_1)(n)_{\geq 0}, (\omega \pi (M/M_1))(n)_{\geq 0}\in \lin A$, we have $(\omega \pi M)(n)_{\geq 0}\in \lin A$ by Lemma \ref{lem.liex}.
\end{proof} 

 Let $A$ be a Koszul algebra.  Then $A$ is the polynomial algebra in $n$ variables if and only if $A^!$ is the exterior algebra in $n$ variables if and only if $\cP^!(A^!)=(\PP^{n-1}, \id)$.  The following lemma is well-known. 

\begin{lemma} \label{lem.qpa} 
If $A$ is an $n$-dimensional quantum polynomial algebra such that $\cP^!(A^!)=(\PP^{n-1}, \t)$, then there exists an equivalence functor $\GrMod A\to \GrMod k[x_1, \dots, x_n]$ which induces a bijection from the set of isomorphism classes of point modules over $A$ to the set of isomorphism classes of point modules over $k[x_1, \dots, x_n]$.  
\end{lemma} 

The following lemma may be standard in commutative algebra.
For the reader's convenience, we include our proof.  

\begin{lemma} \label{lem.inde} 
Let $A=k[x_1, \dots, x_n]$ and $M=A/(x_1^r, x_2, \dots, x_{n-1})\in \grmod A$ for $r\in \NN^+$.  For every $m\in \NN$, $M_{\geq m}$ is indecomposable as a graded $A$-module. 
\end{lemma} 

\begin{proof} If $A'=A/(x_2, \dots, x_{n-1})\cong k[x_1, x_n]$, then $M$ can be viewed as $M'=A'/(x_1^r)$ as a graded $A'$-module.  If $M_{\geq m}$ is decomposable as a graded $A$-module, then $M'_{\geq m}$ is decomposable as a graded $A'$-module, so we may assume that $n=2$.  To simplify the notation, let $A=k[x, y]$ and $M=A/(x^r)$.  If $M_{\geq m}$ is decomposable, then $M_{\geq m'}$ is decomposable for every $m'\geq m$, so we may assume that $m\geq r-1$.  In this case, since $M_m$ is spanned by $x^{r-1}y^{m-r+1}, x^{r-2}y^{m-r+2}, \dots, xy^{m-1}, y^m$ over $k$, every $0\neq u\in M_m$ can be written as  $u=\sum _{i=1}^ja_ix^{r-i}y^{m-r+i}\in M_m$ for some $1\leq j\leq r$ such that $a_j\neq 0$ (and $a_{j+1}=\cdots =a_r=0$).  For $x^{j-1}y^{r-j}\in A_{r-1}$, since $r-i+j-1\geq r$ for every $i\leq j-1$, 
$$ux^{j-1}y^{r-j}=\sum _{i=1}^ja_ix^{r-i+j-1}y^{m+i-j}=a_jx^{r-1}y^m\in M_{m+r-1}.$$
It follows that $x^{r-1}y^m\in uA$ for every $0\neq u\in M_m$.  Suppose that there exists a nontrivial decomposition $M_{\geq m}=L\oplus N$ of $M$.  Since $M_{\geq m}=M_mA$, we see that $L_m\neq 0, N_m\neq 0$, so there must exist $0\neq u\in L_m\subset M_m, 0\neq v\in N_m\subset M_m$ such that $uA\cap vA=0$.  Since $x^{r-1}y^m\in uA\cap vA$ by the above argument, we have a contradiction, so $M_{\geq m}$ is indecomposable. 
\end{proof} 

\begin{lemma} \label{lem.pinf}
Let $A$ be an $n$-dimensional quantum polynomial algebra such that $\cP^!(A^!)=(\PP^{n-1}, \t)$.  For every $r\in \NN^+$ and every $p\in \PP^{n-1}$, there exist an indecomposable $r$-extension of $M_p$. 
\end{lemma} 

\begin{proof} 
By Lemma \ref{lem.qpa}, we may assume that $A=k[x_1, \dots, x_n]$.  
Without loss of generality, we may also assume that $p=(0, \dots, 0, 1)\in \PP^{n-1}$ so that $M_p=A/(x_1, \dots, x_{n-1})$.  Then $M:=A/(x_1^r, x_2, \dots, x_{n-1})$ is an $r$-extension of shifted point modules having the filtration 
$$M=M_0\supset M_1\supset \cdots \supset M_r=0$$
such that $M_i/M_{i+1}\cong M_p(-i)$ for every $i=0, \dots, r-1$.  Since $\omega \pi :\GrMod A\to \GrMod A$ is a left exact functor, we have a filtration 
$$(\omega \pi M)_{\geq 0}=(\omega \pi M_0)_{\geq 0}\supset (\omega\pi M_1)_{\geq 0}\supset \cdots \supset (\omega \pi M_r)_{\geq 0}=0.$$
Since $M_i$ is an $(r-i)$-extension of shifted point modules, we have
$$(\omega \pi M_i)_{\geq 0}/(\omega \pi M_{i+1})_{\geq 0}\cong (\omega \pi (M_i/M_{i+1}))_{\geq 0}\cong (\omega \pi M_p(-i))_{\geq 0}\cong M_p$$
for every $i=0, \dots, r-1$ by the proof of Proposition \ref{prop.fpm} and \cite[Proposition 6.6]{Mcfk}, so $(\omega \pi M)_{\geq 0}$ is an $r$-extension of $M_p$.   

If $(\omega \pi M)_{\geq 0}$ is decomposable, then 
$M_{\geq m}\cong (\omega\pi M)_{\geq m}$ is decomposable for some $m\gg 0$, which is not the case by Lemma \ref{lem.inde}, so $(\omega \pi M)_{\geq 0}$ is indecomposable.  
\end{proof} 

\begin{theorem} \label{thm.nmfext}
Let $T=k\<x_1, \dots, x_n\>/(\a_{ij}x_ix_j+x_jx_i)$ be a skew polynomial algebra, and put $S=T/(x_1^2, \dots, x_{n-1}^2)$.
Let $f=x_n^2\in S_2$ so that $A=S/(f)$ is a skew exterior algebra. Suppose that
$\a_{ij}\a_{jk}\a_{ki}=1$ for all $1\leq i, j ,k\leq n$ (e.g. $A$ is the exterior algebra). 
For each $r\in \NN^+$, there exist infinitely many isomorphism classes of indecomposable $\phi\in \NMF^0_S(f)$ of rank $r$.
\end{theorem} 

\begin{proof}
By \cite[Theorem 4.1]{Ue}, if $\a_{ij}\a_{jk}\a_{ki}=1$ for all $1\leq i, j ,k\leq n$, then $A^!$ is a skew polynomial algebra such that $\cP^!(A)=(\PP^{n-1}, \t)$,
so $A^!$ is an $n$-dimensional quantum polynomial algebra and there exists an indecomposable $r$-extension $M\in \lin A^!$ of the point module $M_p$ over $A^!$ for each $p\in \PP^{n-1}\setminus \cV(x_n)$ by Lemma \ref{lem.pinf}.
Since $E(M)\in \lin A$ is an indecomposable $r$-extension of a co-point module $N_p$ over $A$ where $p\in \PP^{n-1}\setminus \cV(x_n)$,
we have $E(M)\in \TR^0_S(A)$ by Theorem \ref{thm.copo},
so there exists an indecomposable $\phi\in \NMF^0_S(f)$ of rank $r$ such that $\Coker \phi\cong E(M)$.   

For $p, p'\in \PP^{n-1}\setminus \cV(x_n)$, let $M, M'$ be $r$-extensions of $M_p, M_{p'}$, respectively. 
Reducing to the commutative case by Lemma \ref{lem.qpa}, if $p\neq p'$, then 
$M\not \cong M'$, so there exist infinitely many isomorphism classes of indecomposable $\phi\in \NMF^0_S(f)$ of rank $r$.
\end{proof} 

\begin{remark} Let $A$ be a Koszul algebra.  Since the functor $E:\lin A\to \lin A^!$ is a duality, $0\to L\to M\to N\to 0$ is an exact sequence in $\lin A$ if and only if $0\to E(N)\to E(M)\to E(L)\to 0$ is an exact sequence in $\lin A^!$. In fact, the exact sequence $0\to L\to M\to N\to 0$ induces an exact sequence 
$$\begin{array}{cccccccccccc} \Ext^{i-1}_A(L, k)_{-i} & \to & \Ext^i_A(N, k)_{-i} & \to & \Ext^i_A(M, k)_{-i}& \to & \Ext^i_A(L, k)_{-i} & \to & \Ext^{i+1}_A(N, k)_{-i} \\
\parallel & & \parallel & & \parallel & & \parallel & & \parallel \\
0 & &  E(N)_i & & E(M)_i & & E(L)_i & & 0 \end{array}$$ 
for every $i\in \ZZ$.  It follows that $N\in \lin A$ is an (indecomposable) extension of co-point modules $N_p$ if and only if $E(N)$ is an (indecomposable) extension of point modules $M_p$. 
\end{remark} 

\subsection{Extensions of Co-point Modules}

\begin{proposition} \label{prop.extp} 
Let $A$ be an $n$-dimensional quantum polynomial algebra such that $\cP^!(A^!)=(E, \t)$.  Suppose that either 
\begin{enumerate}
\item{} $E=\PP^{n-1}$, or 
\item{} $n=3$ and $||\t||:=\inf\{i\in \NN^+\mid \t^i\in \Aut \PP^{n-1}\}=\infty$.
\end{enumerate}
For $M\in \grmod A$, if $(\omega \pi M)_{\geq 0}\cong M$ and  $H_M(t)=r/(1-t)$, then $M$ is an $r$-extension of point modules. 
\end{proposition} 

\begin{proof} 
The assumption on $A$ implies that 
the isomorphism classes of simple objects in $\tails _0A:=\{\pi M\in \tails A\mid \GKdim M\leq 1\}$ are given by $\{\pi M_p\}_{p\in E}$ by \cite [Lemma 3.5, Proposition 4.4]{Mrb}.  We prove it by induction on $r$.  

Suppose that $(\omega \pi M)_{\geq 0}\cong M$ and  $H_M(t)=1/(1-t)$.  Since $0\neq \pi M\in \tails _0A$, there exists $p\in E$ such that $\pi M_p\subset \pi M$.  Since $\omega :\Tails A\to \GrMod A$ is a left exact functor, $M_p\cong (\omega \pi M_p)_{\geq 0}\subset (\omega \pi M)_{\geq 0}\cong M$.  Since $H_M(t)=1/(1-t)=H_{M_p}(t)$, it follows that $M\cong M_p$ is a point module.  

Suppose that $(\omega \pi M)_{\geq 0}\cong M$ and  $H_M(t)=r/(1-t)$.  Since $0\neq \pi M\in \tails _0A$, there exist $p\in E$ 
and an exact sequence 
$$0\to \pi M_p\to \pi M\to \cF\to 0$$
in $\tails A$ for some $\cF\in \tails A$.  Since $\R^1\omega \pi M_p=\H_{\fm}^2(M_p)=0$, we have an exact sequence 
$$0\to (\omega \pi M_p)_{\geq 0}\cong M_p\to (\omega \pi M)_{\geq 0}\cong M\to (\omega \cF)_{\geq 0}\to (\R^1\omega \pi M_p)_{\geq 0}=0.$$ 
If $M':=(\omega \cF)_{\geq 0}\cong M/M_p\in \grmod A$, then 
$$(\omega \pi M')_{\geq 0}=(\omega \pi (\omega \cF)_{\geq 0})_{\geq 0}\cong(\omega \pi \omega \cF)_{\geq 0}\cong (\omega \cF)_{\geq 0}=M'$$ 
and $H_{M'}(t)=H_M(t)-H_{M_p}(t)=(r-1)/(1-t)$, so $M'$ is an $(r-1)$-extension of point modules by induction, hence $M$ is an $r$-extension of point modules.
\end{proof} 

\begin{proposition} \label{prop.extp2} 
Let $A$ be a Koszul algebra such that $\cP^!(A)=(E, \t)$ and $A^!$ is an $n$-dimensional quantum polynomial algebra.  Suppose that either 
\begin{enumerate}
\item{} $E=\PP^{n-1}$, or 
\item{} $n=3$ and $||\t||=\infty$.
\end{enumerate}
If $M\in \lin A$ has no free summand and  $H_{E(M)}(t)=r/(1-t)$, then $M$ is an $r$-extension of co-point modules. 
\end{proposition}

\begin{proof} 
It is enough to show that $(\omega \pi E(M))_{\geq 0}\cong E(M)$ by Proposition \ref{prop.extp}.   Consider the exact sequence 
$$0\to \H_{\fm}^0(E(M))\to E(M)\to \omega \pi E(M)\to \H_{\fm}^1(E(M)).$$
Since $E(M)\in \lin A^!$, it follows that
$\H_{\fm}^0(E(M))_{\geq 1}=\H_{\fm}^1(E(M))_{\geq 0}=0$ by \cite [Corollary 4.9, Theorem 5.4]{Mr}, so we have an exact sequence 
$$0\to \H_{\fm}^0(E(M))_0\cong k^m\to E(M)\to (\omega \pi E(M))_{\geq 0}\to 0$$
for some $m\in \NN$.  
Since $k^m, E(M)\in \lin A^!$, we have a surjection $M\cong E(E(M))\to E(k^m)\cong A^m$.
Since $M$ has no free summand, $m$ must be $0$, so $(\omega \pi E(M))_{\geq 0}\cong E(M)$.  
\end{proof}  

\begin{lemma} \label{lem.gs} Let $A$ be an $n$-dimensional quantum polynomial algebra such that $\cP^!(A^!)=(\PP^{n-1}, \t)$.  
For point modules $M, M'\in \grmod A$, $\Ext^1_{\GrMod A}(M, M')\neq 0$ if and only if $M\cong M'$.  
\end{lemma} 

\begin{proof} 
By Lemma \ref{lem.qpa}, we may assume that $A=k[x_1, \dots, x_n]$.  Let $M=M_p, M'=M_{p'}$ for $p, p'\in \PP^{n-1}$.  Note that $M\cong M'$ if and only if $p=p'$.  Without loss of generality, we may assume that $p=(1, 0 \dots, 0)$ so that $M=A/(x_2, \dots, x_{n})A$. 
The minimal free resolution of $M$ starts with
$$\begin{CD} A(-2)^{(n-1)(n-2)/2} @>{\left(\begin{smallmatrix} 
x_n & 0 & 0 & \cdots \\
0 & x_n & 0 & \cdots \\
0 & 0 & x_n & \cdots \\
\vdots & \vdots & \vdots & \\
-x_2 & -x_3 & -x_4 & \cdots 
\end{smallmatrix}\right)\cdot }>>  A(-1)^{n-1} @>{\left(\begin{smallmatrix} x_2 & x_3 & \cdots & x_{n}\end{smallmatrix}\right) \cdot }>> A \to M \to 0,\end{CD}$$
so $\Ext^1_{\GrMod A}(M, M')$ is the degree 0 part of the homology of 
$$\begin{CD} M' @>{\cdot \left(\begin{smallmatrix} x_2 & x_3 & \cdots & x_n\end{smallmatrix}\right)}>\phi> M'(1)^{n-1} @>{\cdot \left(\begin{smallmatrix}
x_n & 0 & 0 & \cdots \\
0 & x_n & 0 & \cdots \\
0 & 0 & x_n & \cdots \\
\vdots & \vdots & \vdots & \\
-x_2 & -x_3 & -x_4 & \cdots 
\end{smallmatrix}\right)}>\psi> M'(2)^{(n-1)(n-2)/2}.\end{CD}$$

If $p=p'$, then $\phi=\psi=0$, so $\Ext^1_{\GrMod A}(M, M')=M'(1)^{n-1}_0=(kx_1)^{n-1}\neq 0$.  
If $p\neq p'$, then, without loss of generality, we may assume that $p'=(0, \dots, 0, 1)$ so that $M'=A/(x_1, \dots, x_{n-1})A$.  
In this case, 
$$(\Im \phi)_0=\{(0, \dots, 0, ax_n)\in (kx_n)^{n-1}
=M'(1)_0^{n-1}\mid a\in k\}=(\Ker \psi)_0,$$ 
so $\Ext^1_{\GrMod A}(M, M')=0$.  
\end{proof} 

\begin{theorem} \label{thm.last} 
Let $T=k\<x_1, \dots, x_n\>/(\a_{ij}x_ix_j+x_jx_i)$ be a skew polynomial algebra, and put $S=T/(x_1^2, \dots, x_{n-1}^2)$.
Let $f=x_n^2\in S_2$ so that $A=S/(f)$ is a skew exterior algebra, and let $\cP^!(A)=(E, \t)$.
Suppose that either 
\begin{enumerate}
\item{} $\a_{ij}\a_{jk}\a_{ki}=1$ for all $1\leq i, j ,k\leq n$, or 
\item{} $n=3$ and $\a_{12}\a_{23}\a_{31}$ is not a root of unity.
\end{enumerate}
Then $M\in \TR^0_S(A)$ has no free summand if and only if $M$ is a finite extension of co-point modules $N_{p_i}$ over $A$ where $p_i\in E\setminus \cV(x_n)$. 
\end{theorem} 

\begin{proof} Note that if (1) $\a_{ij}\a_{jk}\a_{ki}=1$ for all $1\leq i, j ,k\leq n$, then $E=\PP^{n-1}$ by \cite[Theorem 4.1]{Ue}, and if (2) $n=3$ and $\a_{12}\a_{23}\a_{31}$ is not a root of unity, then $||\t||=\infty$ by \cite [Lemma 4.13]{Mrb}.  In either case, $A^!$ is a quantum polynomial algebra by Example \ref{ex.qpa}.

If $M\in \TR^0_S(A)$ has no free summand, then there exists $\{\phi^i:F^{i+1}\to F^i\}_{i \in \ZZ}\in \NMF^0_S(f)$ of some rank $r$ such that $\Coker \phi\cong M$ and 
$\overline {F^i}\cong A^r(-i)$ for every $i\in \ZZ$ by Proposition \ref{prop.mfr}. Since $M\in \lin A$ by Proposition \ref{prop.lin},  
$H_{E(M)}=r/(1-t)$, 
so $E(M)$ is an $r$-extension of point modules $M_{p_i}$ for some $p_i\in E$ over $A^!$ by Proposition \ref{prop.extp2}, hence $M$ is an $r$-extension of co-point modules $N_{p_i}$ over $A$.  
We will show that $p_i\not \in \cV(x_n)$ by induction on $r$.  The case $r=1$ follows from Theorem \ref{thm.copo} (1).  Suppose that $M'$ is an extension of co-point modules $N_{p_1}, \dots, N_{p_{r-1}}$ such that there exists an exact sequence 
$$0\to M'\to M\to N_{p_r}\to 0.$$  
By induction, $p_1, \dots, p_{r-1}\in E\setminus \cV(x_n)$.  Since $\TR^0_S(A)$ is closed under direct summand, if the above exact sequence splits, then $N_{p_r}\in \TR^0_S(A)$, so $p_r\in E\setminus \cV(x_n)$.  On the other hand, if the above exact sequence does not split, then the exact sequence $0\to E(N_{p_r})=M_{p_r}\to E(M)\to E(M')\to 0$ does not split.
Since $E(M')$ is an extension of point modules $M_{p_1}, \dots, M_{p_{r-1}}$, we have $\Ext^1_{\GrMod A}(M_{p_i}, M_{p_r})\neq 0$ for some $i=1, \dots, r-1$.  

If (1) $E=\PP^{n-1}$, then 
$\Ext^1_{\GrMod A}(M_{p_i}, M_{p_r})\neq 0$ implies $p_r=p_i\in E\setminus \cV(x_n)$ by Lemma \ref{lem.gs}.
If (2) $n=3$ and $||\t||=\infty$, then either $p_r=p_i\in E\setminus \cV(x_n)$ or $p_r=\varphi ^{-1}\t^{-3}(p_i)$ where $\varphi$ is the Nakayama automorphism of $A^!$ by \cite [Lemma 2.16]{Aj} (cf. \cite [Proposition 4.22]{Mrb}).  Since $\t(\cV(x_n))=\cV(x_n)$ and $\varphi (\cV(x_n))=\cV(x_n)$, we have $\varphi ^{-1}\t^{-3}(p_i)\in E\setminus \cV(x_n)$ by \cite [Theorem 4.1]{Ue}, hence the result. 

Conversely, if $M$ is a finite extension of co-point modules $M_{p_i}$ over $A$ where $p_i\in E\setminus \cV(x_n)$, then $M\in \TR^0_S(A)$ by Theorem \ref{thm.copo} (2).  If $M$ has a free  summand, then $E(M)$ is not torsion-free.   However, since $E(M)$ is a finite extension of point modules, $E(M)$ is torsion-free, so $M$ has no free summand. 
\end{proof}


\begin{thebibliography}{99}

\bibitem{Aj}
K. Ajitabh, Modules over elliptic algebras and quantum planes, 
\textit{Proc. London Math. Soc. (3)} \textbf{72} (1996), no. 3, 567--587.

\bibitem{AZ}
M. Artin and J. J. Zhang,
Noncommutative projective schemes,
\textit{Adv. Math.} \textbf{109} (1994), no. 2, 228--287.

\bibitem{AM}
L. L. Avramov and A. Martsinkovsky,
Absolute, relative, and Tate cohomology of modules of finite Gorenstein dimension,
\textit{Proc. London Math. Soc. (3)} \textbf{85} (2002), no. 2, 393--440.

\bibitem{CCKM}
T. Cassidy, A. Conner, E. Kirkman, and W. F. Moore,  
Periodic free resolutions from twisted matrix factorizations,
\textit{J. Algebra} \textbf{455} (2016), 137--163. 

\bibitem{Ei}
D. Eisenbud,
Homological algebra on a complete intersection, with an application to group representations,
\textit{Trans. Amer. Math. Soc.} \textbf{260} (1980), no. 1, 35--64.

\bibitem{KKZ}
E. Kirkman, J. Kuzmanovich and J. J. Zhang,
Noncommutative complete intersections,
\textit{J. Algebra} \textbf{429} (2015), 253--286.

\bibitem{L}
T. Levasseur,
Some properties of noncommutative regular graded rings,
\textit{Glasgow Math. J.} \textbf{34} (1992), no. 3, 277--300.

\bibitem{LS}
T. Levasseur and J. T. Stafford,
The quantum coordinate ring of the special linear group,
\textit{J. Pure Appl. Algebra} \textbf{86} (1993), no. 2, 181--186. 

\bibitem{Mr}
I. Mori,
Rationality of the Poincar\'e series for Koszul algebras, 
\textit{J. Algebra} \textbf{276} (2004), no. 2, 602--624. 

\bibitem{Mck}
I. Mori,
Co-point modules over Koszul algebras,
\textit{J. London Math. Soc. (2)} \textbf{74} (2006), no. 3, 639--656.

\bibitem{Mcfk} 
I. Mori, 
Co-point modules over Frobenius Koszul algebras, 
\textit{Comm. Algebra} \textbf{36} (2008), no. 12, 4659--4677.

\bibitem{Mrb}
I. Mori, 
Regular modules over 2-dimensional quantum Beilinson algebras of Type S, 
\textit{Math. Z.} \textbf{279} (2015), no. 3--4, 1143--1174.

\bibitem{MU}
I. Mori and K. Ueyama,
A categorical characterization of quantum projective spaces,
\textit{J. Noncommut. Geom.}, to appear, \texttt{arXiv:1708.00167}.

\bibitem{MU2}
I. Mori and K. Ueyama,
Noncommutative Kn\"orrer's periodicity theorem and noncommutative quadric hypersurfaces,
\textit{Algebra Number Theory}, to appear, \texttt{arXiv:1905.12266}.

\bibitem{SV}
S. P. Smith and M. Van den Bergh,
Noncommutative quadric surfaces,
\textit{J. Noncommut. Geom.} \textbf{7} (2013), no. 3, 817--856.

\bibitem{SZ}
D. R. Stephenson and J. J. Zhang,
Growth of graded Noetherian rings,
\textit{Proc. Amer. Math. Soc.} \textbf{125} (1997), no. 6, 1593--1605. 

\bibitem{Ue} 
K. Ueyama, 
Graded Morita equivalences for generic Artin-Schelter regular algebras,
\textit{Kyoto J. Math.} \textbf{51} (2011),  no. 2, 485--501.

\bibitem{Y}
Y. Yoshino,
\textit{Cohen-Macaulay modules over Cohen-Macaulay rings},
London Mathematical Society Lecture Note Series, 146, Cambridge University Press, Cambridge, 1990.
 
\bibitem{Zh}
J. J. Zhang,
Twisted graded algebras and equivalences of graded categories,
{\it Proc. Lond. Math. Soc. (3)} {\bf 72} (1996), no. 2, 281--311. 
\end{thebibliography}
\end{document}